\documentclass[final,leqno]{siamltex}

\usepackage{amsfonts}
\usepackage{amsmath}
\usepackage{amssymb}
\usepackage{bbm}
\usepackage{graphicx}
\usepackage{float}
\usepackage{subfigure}
\usepackage{booktabs}
\usepackage{setspace}
\usepackage[toc,page]{appendix}
\usepackage{enumitem}
\usepackage{comment}
\usepackage{multicol}
\usepackage{multirow}
\usepackage{color}
\usepackage[table]{xcolor}
\usepackage{url}

\DeclareMathOperator*{\esssup}{ess\,sup}

\DeclareMathOperator*{\essinf}{ess\,inf}

\newtheorem{assumption}[theorem]{Assumption}
\newtheorem{example}[theorem]{Example}
\newtheorem{remark}[theorem]{Remark}

\newcommand{\sX}{{\mathcal X}}
\newcommand{\R}{{\mathbb R}}

\usepackage{color}

\newcommand{\bDGH}[1]{{\color{black}{#1}}}

\newcommand{\E}{\mathbb E}
\newcommand{\Prob}{\mathbb P}

\newcommand{\1}{\mathbbm{1}}
\newcommand{\ol}[1]{{\overline{#1}}}
\newcommand{\ul}[1]{{\underline{#1}}}

\title{Zero-sum Dynkin games under common and independent Poisson constraints}

\author{David Hobson,\ \ Gechun Liang,\ \ Edward Wang\thanks{Department of Statistics, University of Warwick, Coventry, CV4 7AL, U.K.
Email address: {\tt d.hobson@warwick.ac.uk;}\ \ {\tt g.liang@warwick.ac.uk;}\ \ {\tt
e.wang.5@warwick.ac.uk}}}
\begin{document}

\maketitle

\begin{abstract}
\bDGH{Zero-sum Dynkin games under Poisson constraints, where players can only stop at the event times of a Poisson process, have been studied widely in the recent literature.}
The constraint can be modelled in two ways: either both players share the same Poisson process (the common constraint) or each player has their own Poisson process (the independent constraint). \bDGH{In a Markovian framework, where payoffs are functions of an underlying diffusion, we establish necessary and sufficient conditions for the equivalence of the game’s solution—comprising the value function and optimal stopping sets—under the common and independent constraints.} Specifically, if the stopping sets of the maximiser and minimiser in the game under the common constraint are disjoint, then the solution to the game is the same under both the common and the independent constraint. However, the fact that the stopping sets are disjoint in the game under the independent constraint, is not sufficient to guarantee that the solution of the game under the independent constraint is also the solution under the common constraint. \bDGH{To demonstrate the broad applicability of our results, we solve infinite-horizon Dynkin games satisfying the assumptions of our main theorems, using backward stochastic differential equation (BSDE) techniques. This requires extending standard BSDE results from the finite-horizon setting to the infinite-horizon case, allowing for unbounded solutions.}
\end{abstract}

\begin{keywords}
{Zero-sum Dynkin game, common Poisson constraint, independent Poisson constraint, backward stochastic differential equation.} 
\end{keywords}

\begin{AMS}  {60G40, 91A05, 49L20.}  
\end{AMS}

\pagestyle{myheadings} \thispagestyle{plain} \markboth{}{Zero-sum Dynkin games}

\section{Introduction}

A zero-sum Dynkin game is a game played by two players, a maximiser and a minimiser, who each choose a stopping time to maximise/minimise the expected payoff. The two players are then denoted as the \textsc{sup} player and the \textsc{inf} player. If {the \textsc{sup} player} is first to stop (at $\tau$) then the payment is $L_\tau$; if {the \textsc{inf} player} is first to stop (at $\sigma$) then the payment is $U_\sigma$; under a tie ($\tau = \sigma$) then the payment is $M_\tau$. Dynkin games were first introduced by Dynkin \cite{dynkin1969game}, and later extended in a discrete time setup to the above form by Neveu \cite{neveu1975discrete}, and in a continuous time setup by Bismut \cite{bismut1977probleme}. Since then, there have been multiple works on discrete time or continuous time Dynkin games. For example, Cvitanic and Karatzas \cite{cvitanic1996backward} utilized a backward stochastic differential equation (BSDE) approach, Ekstr{\"o}m and Peskir \cite{Ekstrom2008SICON} considered Dynkin games under a Markovian setup (in the sense that the payoffs are disounted functions of an underlying diffusion) and Kifer \cite{Kifer_2} gave an extensive survey with applications to financial game options. An extension of zero-sum Dynkin games allows for randomized strategies. For relevant literature on this topic, refer to Laraki and Solan \cite{laraki2005value}, Rosenberg et al \cite{rosenberg2001stopping}, and Touzi and Vieille \cite{touzi2002continuous}. {More recent work in this area includes De Angelis et al \cite{DeAngelis2020, DeAngelis2021_a, DeAngelis2021_b}} for asymmetric information  and Dai and Dong \cite{DaiDong2024} for applications to reinforcement learning.

In this paper, we consider Dynkin games with an extra constraint on the players' stopping strategies. The idea is that both players may only stop when they receive a signal generated by a Poisson process which is independent of the payoff processes. Optimal stopping problems with this type of constraint were first studied by Dupuis and Wang \cite{dupuis2002optimal} and further developed in Lempa \cite{lempa2012optimal}, Liang \cite{liang2015stochastic},  Liang et al \cite{Liang_2025}, Hobson \cite{Hobson}, Hobson and Zeng \cite{HZ} and {Alvarez et al \cite{Alvarez2024}}. For Dynkin games, P{\'e}rez et al \cite{perez2021non} applied the constraint to one of the players but most of the literature on Dynkin games under constrained stopping focuses on the case where both players are constrained. This has been done by following one of two approaches: either let both players use a common signal process or assign independent signal processes to each player. Under the first approach the players share the constraint, and we denote this approach as the `common Poisson constraint' case. 
Liang and Sun \cite{Liang_Sun} analysed Dynkin games with the common Poisson constraint under the `order condition' $L= M\leq U$ and proved the existence of a saddle point via a BSDE approach, see also Hobson et al \cite{hobson2024callable} for an example under the `generalised order condition' $L\wedge U\leq M\leq L\vee U$. (This condition was also used in the studies by Stettner \cite[Theorem 3]{stettner1982}, Merkulov \cite[Chapter 5]{merkulov2021value} and Guo \cite{guo2020dynkin}
of the existence of the game value in the unconstrained case.)  Under the second approach the Poisson processes describing the constraints for the two players are independent and we denote this approach the `independent Poisson constraint'. Dynkin games under an independent Poisson constraint are studied in Du et al \cite{liang2020risk}, Lempa and Saarinen \cite{lempa2023zero} and {Gapeev \cite{Gapeev2024}.}

In this paper we consider a perpetual zero-sum Dynkin game in a Markovian setting in which both players are constrained to choose stopping times taking values in the event times of a Poisson process.
The main problem we investigate is whether the Dynkin game is the same in the common and the independent constraint setups (in the sense that the value function and the stopping sets of each player do not depend on whether we work in the common or independent constraint problem). Our study is motivated by a finite-horizon example presented in \cite[Remark 5.2]{liang2020risk} in which the game value and optimal strategies of the two players are the same in the two setups, and we wanted to understand the extent to which this result is true in general.

In the main setting of our study the payoffs are based on discounted functions of a time-homogeneous diffusion $X$. In particular $L$ is given by $L_t= e^{-rt}l(X_t)$ and $U$ is given by $U_t = e^{-rt} u(X_t)$ for a pair of non-negative functions $l$ and $u$ defined on the state space of $X$, and we take $M= L$. \bDGH{(In Appendix~\ref{app:m=u}, we consider how the results are modified under the choice $M= U$, and the extent to which the results generalise to other choices of $M$.)} We consider two constrained optimal stopping games, in one each player can only stop at event times of a common Poisson process (of rate $\lambda>0$), and in the other each player is constrained to only stop at event times of their own, individual Poisson process (where this pair of Poisson processes are independent, and each is of rate $\lambda$). In each of these settings (common constraint and independent constraint) it is reasonable to expect that the value function is a \bDGH{continuous} function of the initial value of the underlying diffusion, and the optimal stopping rules are for the \textsc{sup} player to stop at the first event time of their Poisson process (or, in the common constraint problem, at the first event time of the common Poisson process) at which the diffusion is in some (time-independent) set $A$, and for the \textsc{inf} player to stop at the first event time of their Poisson process (or, in the common constraint problem, at the first event time of the common Poisson process) at which the diffusion is in some other (time-independent) set $B$. We give results (Theorems~\ref{thm:satisfyassC} and \ref{thm:satisfyassi}) to show that under some technical assumptions the solution of the game does indeed take this form.
{Indeed, if \( v^I \) is the value of the game under the independent constraint (expressed as a function of the initial value), then an optimal stopping set for the \textsc{sup} player is \( A^I = \{ x : v^I(x) < l(x) \} \), while an optimal stopping set for the \textsc{inf} player is \bDGH{\( B^I = \{ x : v^I(x) > u(x) \} \)}.
Furthermore, if \( v^C \) denotes the value of the game under the common constraint (also as a function of the initial value), an optimal stopping set for the \textsc{sup} player can be chosen as
\(
A^C = \{ x : v^C(x) < l(x) \} \cup \bDGH{\{ x : v^C(x) \geq l(x) > u(x) \}}
\)
and an optimal stopping set for the \textsc{inf} player as
\(
B^C = \bDGH{\{ x : v^C(x) > u(x) \}}.
\)}
Note the difference in form of the stopping set for the \textsc{sup} player: in the common constraint problem the \textsc{sup} player additionally stops if the \textsc{inf} player intends to stop and the \textsc{sup} player can get a higher payoff by stopping themselves.

Suppose that the solution to the Dynkin game under the common constraint is of the form $(v^C, A^C, B^C)$. We ask: under what conditions is this also the solution of the game under the independent constraint. \bDGH{Under the assumption that $M = L$, a necessary and sufficient condition is that ${v^C \wedge l \leq u}$.} Note that under the order condition $l \leq u$ this is automatically satisfied, \bDGH{but under the generalised order condition this may not be the case.} 
The condition \bDGH{$v^C \wedge l \leq u$} may be written as $A^C$ and $B^C$ are disjoint. Now, suppose that the solution to the Dynkin game under the {independent} constraint is of the form $(v^I, A^I, B^I)$. We ask: under what conditions is this also the solution of the game under the common constraint. A  \bDGH{necessary and sufficient condition} is that ${v^I \wedge l \leq u}$, and again, under the order condition this is always satisfied. However, \bDGH{the condition ${v^I \wedge l \leq u}$} is {\em not} equivalent to the fact that $A^I$ and $B^I$ are disjoint.

\bDGH{As described above, under the order condition $l \leq m \leq u$ the condition $v^C \wedge l \leq u$ is automatically satisfied. 
Hence, under the order condition the results are very simple: the value of the game under the common constraint is the same as the value under the independent constraint, and the stopping sets also correspond. The reason for this is that under the order condition it is never optimal for both players to want to stop at the same time. Instead, the real content of our result lies in the case where the order condition is not satisfied. Then, in principle it matters whether the Poisson events are potential stopping events for both players, or just one. The key point is that if $M = L$ then under the condition $v^C \wedge l \leq u$ it is never optimal for both players to want to stop at the same time. Under the independent model, with two Poisson processes each of rate $\lambda$, one for each player, the effective Poisson rate at which either player stops is $\lambda$ on the (disjoint) union of the two players' stopping sets; this is also the rate of stopping in the common Poisson model.}

The rest of the paper is structured as follows. In Section~\ref{sec:setup} we introduce Dynkin games under the common and the independent constraints in a diffusion setting. Our main results on constrained zero-sum Dynkin games are presented in Section~\ref{sec:translate}. Section~\ref{sec:examples} contains two examples. In the first example, the optimal stopping sets
for the problem under the common constraint are disjoint, and therefore the value function and stopping sets of the independent constraint game are the same as those for the common constraint. Then, in a second example we show that the converse is not true---in this example we give the game value and (disjoint) stopping sets for the game under the independent constraint but these do not define the solution for the game under the common constraint.
Finally, in Section~\ref{sec:BSDE}, we prove existence results for zero-sum Dynkin games in the diffusion setting under the common and the independent constraints. This allows us to give a large family of problems for which we know that there exist solutions (with time-homogeneous stopping sets) to the game with either the common or independent constraint, at which point our main result may be applied.
A key element of our modelling is that we assume an infinite horizon problem with unbounded payoffs. This means that the standard finite-horizon BSDE methodology of, for example, Du et al ~\cite{liang2020risk} does not apply. \bDGH{To deal with the infinite horizon and unbounded solutions, we utilize an approximation argument and introduce exponential weighting into the norms on the solution spaces to control growth at infinity. Results concerning alternative choices of $M$ and technical proofs are provided in appendices.} 

\section{Setup}
\label{sec:setup}


We work on a filtered probability space $(\Omega,\mathcal{F}, \mathbbm{F}=(\mathcal{F}_t)_{t\geq 0}, \mathbb{P})$ satisfying the usual conditions, 
and consider a zero-sum Dynkin game. 
The payoff process for the \textsc{sup} player is given by $L= (L_t)_{t \geq 0}$ (where $L$ is a progressively measurable process), and we denote their stopping time by $\tau$, so that if they stop first the payoff of the game is $L_\tau$. Similarly, the payoff process for the \textsc{inf} player is given by the progressively measurable process $U= (U_t)_{t \geq 0}$, and we denote their stopping time by $\sigma$, so that if they stop first the payoff of the game is $U_\sigma$. \bDGH{If both players stop simultaneously ($\tau = \sigma$) then the payoff is given by the progressively measurable process $M=(M_t)_{t \geq 0}$. }
It is quite common in the literature to assume the order condition $L \leq M\leq U$. 
If both players stop at the same time ($\tau = \sigma$) we assume that the stopping of the \textsc{sup} player takes precedence (so the payoff \bDGH{on simultaneous stopping} is $M_\tau=L_\tau$). It is very possible to consider other conventions on the set $\tau = \sigma$ (for example, the payoff is $M_\sigma=U_\sigma$) without changing the spirit of our results. \bDGH{In the main body of the paper we assume $M = L$, and we defer discussion of other cases to Appendix \ref{app:m=u}.}
Finally, if neither player stops at a finite time (i.e. $\tau =\infty=\sigma$) we assume that the payoff is given by an $\mathcal{F}_\infty$-measurable random variable $M_\infty$. Candidate examples are $M_\infty = 0$ or $M_\infty = \lim\limits_{t \rightarrow \infty} L_t$ assuming that this limit exists.
In summary, the payoff of the game $R=R(\tau,\sigma)$ is given by
\begin{equation} 
\label{eq:gamepayoff}
    R(\tau,\sigma) 
    = L_\tau\1_{\{\tau\leq \sigma<\infty\}}+U_\sigma\1_{\{\sigma<\tau\}}+M_\infty\1_{\{\tau=\sigma=\infty\}} ,
\end{equation}
and the expected payoff is defined by $J(\tau,\sigma)=\E[R(\tau,\sigma)]$.

In a general Dynkin game $\tau$ and $\sigma$ can be any stopping times. Instead, we introduce constraints to players' stopping strategies by only allowing them to stop when they receive a signal. The signal is modelled by an independent Poisson process. We want to consider modelling this in two ways.

The first approach is to assume that both players have the same signal process, which we denote by the `common Poisson constraint', {or the `common constraint' as shorthand}. In this case, we assume that the filtered probability space $(\Omega,\mathcal{F}, \mathbbm{F}, \mathbb{P})$ supports a Poisson process with constant intensity $\lambda>0$ which is independent of the payoff processes $L$ and $U$ (and terminal random variable $M_\infty$). We denote the jump times of the Poisson process by $\{T_n^C\}_{1\leq n\leq \infty}$, with $T^C_\infty=\infty$. 
We assume $\tau,\sigma\in\mathcal{R}^C(\lambda)$, where
\[{\mathcal{R}^C}(\lambda)=\{ \mbox{$\gamma$: $\gamma$ a stopping time such that $\gamma(\omega)=T^C_n(\omega)$ for some $n\in \{ 1, \dots \infty \}$}\}.\]
The upper and lower values of the Dynkin game under the common Poisson constraint are defined as:
\begin{equation}\label{eq:commonvalues}
    \ol{v}^C  = \inf_{\sigma \in \mathcal{R}^C(\lambda)} \sup_{\tau \in \mathcal{R}^C(\lambda)} J(\tau,\sigma), \hspace{8mm} \ul{v}^C = \sup_{\tau \in \mathcal{R}^C(\lambda)} \inf_{\sigma \in \mathcal{R}^C(\lambda)} J(\tau,\sigma).
\end{equation}
If $\ul{v}^C = \ol{v}^C$ then we say the game has a value $v^C$ where $v^C = \ul{v}^C=\ol{v}^C$. A pragmatic approach to finding a solution is to try to find a saddle point, i.e. to find a pair $(\tau^*,\sigma^*) \in \mathcal{R}^C(\lambda) \times \mathcal{R}^C(\lambda)$ such that $J(\tau, \sigma^*) \leq J(\tau^*,\sigma^*) \leq J(\tau^*, \sigma)$ for all $(\tau,\sigma) \in \mathcal{R}^C(\lambda) \times \mathcal{R}^C(\lambda)$. Then,
\[ \ol{v}^{C} \leq \sup_{\tau \in \mathcal{R}^C(\lambda)} J(\tau,\sigma^*) \leq J(\tau^*,\sigma^*) \leq \inf_{\sigma \in \mathcal{R}^C(\lambda)} J(\tau^*,\sigma) \leq \ul{v}^{C}. \]
Since trivially $\ul{v}^{C} \leq \ol{v}^{C}$, we conclude that $\ul{v}^{C} = \ol{v}^{C}$ and the game has a value.

The second approach is to assign separate signal processes for the two players, and we call this the `independent Poisson constraint' case, {or the `independent constraint' as shorthand}. In this case, we assume that the probability space $(\Omega,\mathcal{F}, \mathbbm{F}, \mathbb{P})$ supports two independent Poisson processes, both\footnote{{Of course, in this case the players may have signal processes with different rates, but we focus on the case where the rates are the same because we want to compare to the common constraint problem.}} with intensity $\lambda$, such that the two Poisson processes are independent of the payoffs $L,U,M_\infty$. We denote the jump times of the Poisson process by $\{T_n^{M}\}_{1\leq n\leq \infty}$ and $\{T_n^{m}\}_{1\leq n \leq \infty}$ respectively, with $T^{M}_\infty=T_{\infty}^{m}=\infty$. 
We assume $\tau\in\mathcal{R}^{M}(\lambda)$ and $\sigma\in\mathcal{R}^{m}(\lambda)$, where
\begin{align*}
    {\mathcal{R}^{M}}(\lambda)&=\{ \mbox{$\gamma$: $\gamma$ a stopping time such that $\gamma(\omega)=T^{M}_n(\omega)$ for some $n\in \{ 1, \dots \infty \}$}\};\\
    {\mathcal{R}^{m}}(\lambda)&=\{ \mbox{$\gamma$: $\gamma$ a stopping time such that $\gamma(\omega)=T^{m}_n(\omega)$ for some $n\in \{ 1, \dots \infty \}$}\}.
\end{align*}
The upper and lower values of the Dynkin game under the independent Poisson constraint are defined as:
\begin{equation}\label{eq:indepvalues}
    \ol{v}^I  = \inf_{\sigma \in \mathcal{R}^{m}(\lambda)} \sup_{\tau \in \mathcal{R}^{M}(\lambda)} J(\tau,\sigma), \hspace{8mm} \ul{v}^I = \sup_{\tau \in \mathcal{R}^{M}(\lambda)} \inf_{\sigma \in \mathcal{R}^{m}(\lambda)} J(\tau,\sigma).
\end{equation}
The value $v^I$ and saddle point of the game can be defined in the same way as in the first approach.

We suppose the filtered probability space $(\Omega,\mathcal{F}, \mathbbm{F}=(\mathcal{F}_t)_{t\geq 0}, \mathbb{P})$ supports a regular linear diffusion process $X$ driven by a one dimensional Brownian motion $W$. Let the \bDGH{open} interval $\sX\subset \mathbb{R}$ be the state space of $X$. 
We assume that the process $X$ does not die inside $\sX$ and its endpoints are not exit points (see \cite[Chapter 2]{borodin1997handbook} for the classification of endpoints). 
The payoff processes $L$ and $U$ are given by discounted functions of the diffusion $X$ so that $L_t = e^{-r t} l(X_t)$ and $U_t = e^{-r t} u(X_t)$ for some $r>0$ and non-negative functions $l$ and $u$ defined on $\sX$. 

Throughout the paper we denote by $\mathbb{P}^x$ the probability measure $\mathbb{P}$ conditioned on the initial state $X_0=x$, and let $\E^x$ be the expectation with respect to $\mathbb{P}^x$, and use superscript $x$ on random variables to denote the initial state if appropriate.
In this setting we consider a family of zero-sum Dynkin games indexed by the initial point $x$ of the diffusion. 
Then \eqref{eq:gamepayoff} (indexed by $x$) becomes
\begin{equation} 
\label{eq:gamepayoffd}
    R^x(\tau,\sigma) = e^{-r\tau}l(X^x_\tau)\1_{\{\tau\leq \sigma<\infty\}}+e^{-r\sigma}u(X^x_\sigma)\1_{\{\sigma<\tau\}}+M^x_\infty\1_{\{\tau=\sigma=\infty\}}
\end{equation}
where the payoff on $\{\tau = \sigma = \infty\}$ may also depend on $x$ (for example, via $X^x_\infty$ if this limit exists). 

\bDGH{For each $x\in \sX$, (\ref{eq:gamepayoffd}) defines a Dynkin game and we can still talk about the game under a Poisson constraint. If the game value exists for some $x$ then we denote the value by $v^C(x)$ or $v^I(x)$.} By extension we can construct a family of game values, one for each initial value of $X$, and therefore define functions $v^C, v^I : \sX \mapsto \R$ representing the game values under the common and independent Poisson constraints respectively.

Finally, we introduce the first hitting time. For arbitrary initial state $X_0=x\in \sX$ and Borel set $D\subset \sX$, define $\eta^C_D=\eta^C_D(x)=\inf\{T^C_n : n \geq 1: X^x_{T^C_{n}}\in D\}$. $\eta^C_D$ is the first event time of the Poisson process (with jump times $\{T_n^C\}_{1\leq n\leq \infty}$) such that the diffusion process $X$ at this time is in the set $D$. Similarly, under the independent Poisson constraint, we can define first hitting times $\eta^{M}_D$, $\eta^{m}_D$ using the corresponding Poisson processes in the same way.

\section{Equivalence of Dynkin games under different constraints}
\label{sec:translate}
In this section we assume that we have the solution to a zero-sum Dynkin game under either the common or independent constraint, and try to give conditions under which this solution is also the solution under the alternative constraint. We assume that the given solution takes a particular time-homogeneous form (see Assumption~\ref{assumption:existence}). If the problem is in the setting of a time-homogeneous diffusion then the strong Markov property, together with the form of the payoffs, mean that we expect that the optimal strategies for the two players will take the forms of first hitting times. 
Under some technical assumptions we prove in Section~\ref{sec:BSDE} below that this is indeed the case for a wide class of examples. But here, we try to make minimal assumptions on the solution to the problem, and try to derive the main results in a general setting. This allows us to focus on the important properties of the game value under which the main results are true.

\subsection{From the common constraint to the independent constraint}
\label{ssec:ctoi}

We make the following pair of assumptions on the Dynkin game under the common constraint throughout this subsection:
\begin{assumption}\label{assumption:existence}
    The Dynkin game under the common Poisson constraint has a value function $\{ v^C(x) \}_{x \in \sX}$, with saddle point $(\eta^C_{A^C},\eta^C_{B^C})$, where 
    \bDGH{ $A^C=\{x:v^C(x)< l(x) \}\cup\{x:v^C(x) \geq l(x)> u(x)\}$, $B^C=\{x:v^C(x)> u(x)\}$.  }
    
\end{assumption}



\begin{assumption}\label{assumption:technical}
    For every $x\in \sX$, the random variable $M_\infty$ and stopping times $\eta_{A^C}^C$, $\eta_{B^C}^C$ satisfy $e^{-r\gamma}\E^{X^x_\gamma}[M_\infty\1_{\eta_{A^C}^C\wedge \eta_{B^C}^C=\infty}]=\E^x[M_\infty\1_{\eta_{A^C}^C\wedge \eta_{B^C}^C=\infty}|\mathcal{F}_{\gamma}]$ for any $\gamma \in \mathcal{R}^C(\lambda)$. Here $A^C$ and $B^C$ are the sets defined in Assumption \ref{assumption:existence}.
\end{assumption}

Later, in Section~\ref{sec:BSDE}, we will provide a set of sufficient conditions on the Dynkin game for Assumption~\ref{assumption:existence} to hold.
In particular, we will argue that the results of this subsection apply to a wide class of problems.

\begin{remark}\label{rem:openclosed}
    \bDGH{In Assumption \ref{assumption:existence} the inequalities in the defintions of $A^C$ and $B^C$ are chosen such that if $v^C$ is continuous, $l$ is lower semicontinuous and $u$ is upper semicontinuous then $A^C$ and $B^C$ are both open, see Lemma~\ref{lem:open}. This choice is useful when we consider whether the sufficient condition we give in 
    \bDGH{Theorem \ref{thm:necessity}} is also necessary (see Section \bDGH{\ref{ssec:necessity}}). 
    Nevertheless, there is no uniqueness of the saddle point at the level of stopping sets; for example if we take $\hat{B}^C = \{ x : v^C(x) \geq u(x) \}$ then $(\eta^C_{A^C},\eta^C_{\hat{B}^C})$ is also a saddle point.} 
\end{remark}


We have already argued (and will later prove) that Assumption~\ref{assumption:existence} is satisfied in a wide class of constrained zero-sum Dynkin games. The next lemma gives several situations in which Assumption~\ref{assumption:technical} is satisfied.

\begin{lemma}
\label{lem:sufftech}
    Assumption \ref{assumption:technical} holds if any of the following conditions is satisfied:
    \begin{enumerate}
        \item $\eta^C_{A^C}\wedge \eta_{B^C}^C$ is finite $\mathbb{P}^x$-a.s. for every $x\in \sX$;
        \item $M_\infty=0$; 
        \item $M_\infty=\lim\limits_{t\rightarrow \infty} e^{-rt}m(X_t)$, for some non-negative function $m$ on $\sX$, where we further assume that the limit exists $\mathbb{P}^x$-a.s. for every $x\in \sX$, and that $\E^x[\sup_{t\geq 0}e^{-rt}m(X_t)]<\infty$ for every $x\in \sX$.
    \end{enumerate}
\end{lemma}
\begin{proof}
    If $\eta^C_A\wedge \eta_B^C$ is finite $\mathbb{P}^x$-a.s. for every $x\in \sX$, then the set $\{\eta^C_A\wedge \eta^C_B=\infty\}$ is a null set regardless of initial condition, hence the required equality follows. Similarly, if $M_\infty=0$, the equality is also trivial.

    In the final case, by strong Markov property:
    \begin{eqnarray*}
        \lefteqn{\E^x[M_\infty\1_{\eta_A^C\wedge \eta_B^C=\infty}|\mathcal{F}_\gamma]}\\
        &=& \E^{x}[ \lim_{t\rightarrow \infty} e^{-r(t+\gamma)}m(X_{t+\gamma})\1_{\eta_A^C\wedge \eta_B^C>t+\gamma}|\mathcal{F}_\gamma]\\
        &=&\lim_{t\rightarrow \infty} e^{-r\gamma}\E^{x}[ e^{-rt}m(X_{t+\gamma})\1_{\eta_A^C\wedge \eta_B^C>t+\gamma}|\mathcal{F}_\gamma]\\
         &=& \lim_{t\rightarrow \infty} e^{-r\gamma}\E^{X^x_\gamma}[ e^{-rt}m(X_{t})\1_{\eta_A^C\wedge \eta_B^C>t}]\\
         &=& e^{-r\gamma}\E^{X^x_\gamma}[ \lim_{t\rightarrow \infty} e^{-rt}m(X_{t})\1_{\eta_A^C\wedge \eta_B^C>t}]=e^{-r\gamma}\E^{X^x_\gamma}[ M_\infty\1_{\eta_A^C\wedge \eta_B^C=\infty}],
    \end{eqnarray*}
where the swap of limit and expectation is justified by the integrability assumption.\end{proof}


\begin{example}
\label{eg:canonical}
    We provide a canonical example where Assumptions~\ref{assumption:technical} and \ref{assumption:existence} hold, with $M_\infty=0$. 

    Let $X$ be a geometric Brownian motion with drift $\mu$ and volatility $\sigma$, such that $\frac{\sigma^2}{2}<r$. Suppose $\mu<\frac{\sigma^2}{2}$. Then $\sX=(0,\infty)$. Suppose $l$ is given by $l(x)=(x-P)^+$ for some $P>0$. Note that $L_\infty := \lim\limits_{t \rightarrow \infty} e^{-rt} l(X_t)$ exists and that $L_\infty=0$. By \cite[Proposition 4.5]{hobson2024callable}, the optimal stopping problem $\sup_{\tau\in \mathcal{R}^C(\lambda)}\E^x[e^{-r\tau}l(X_\tau)]$ has a value $V(x)$ and the optimal stopping time is given by $\eta^C_{[x^*,\infty)}$ for some $x^*>0$, with $V(x)>l(x)$ on $(0,x^*)$ and $V(x)<l(x)$ on $(x^*,\infty)$. Suppose $u$ is given by $u(x)=(l\vee V)(x)+P$ and take $M_\infty=0$. (This is very reasonable in this problem since $L_\infty$ and $U_\infty := \lim\limits_{t \rightarrow \infty} e^{-rt} u(X_t)$ both exist and are equal to zero.)

    It follows from \cite[Theorem 2.6]{hobson2024callable} (and is easy to believe) that we have a saddle point of the form stated in Assumption \ref{assumption:existence}, with $v^C=V$, $A=A^C=[x^*,\infty)$ and $B=B^C=\emptyset$. Given that $\mu<\frac{\sigma^2}{2}$ we have $\mathbb{P}^x(\eta_A^C=\infty)>0$ and $ \mathbb{P}^x(\eta_B^C=\infty)=1$ for every $x\in \sX$. 
\end{example}

\begin{example}\label{eg:extension}
     In this example, $\eta^C_A = \eta^C_B=\infty$ a.s. and $M_\infty\neq 0$ a.s.. Nonetheless, Assumption \ref{assumption:technical} and \ref{assumption:existence} are both satisfied. Let $X$ satisfy $dX_t=rX_t\,dt+dW_t$ with $r>0$ and $X_0=x$, so that $e^{-rt}X_t=x+\int_0^te^{-rs}\,dW_s$ which converges a.s. as $t\uparrow \infty$ to a random variable which has a non-centred Gaussian distribution. Suppose $l(x)=|x|$. Then $L =(L_t)_{t\geq 0}$ given by $L_t = e^{-rt} l(X_t)$ is a square integrable submartingale for any initial value $x$ with limit $L_\infty := \lim\limits_{t \rightarrow \infty} L_t$ and $V(x) :=\sup_{\tau\in \mathcal{R}(\lambda)}\E^x[L_\tau]=\E^x[L_\infty] = \E[|G_{x,1/2r}|]$ where $G_{a,b}$ is a Gaussian random variable with mean $a$ and variance $b$.

     Choose $c> \E[|G_{0, 1/2r}|]$. Set $u(x):=l(x)+c$ and note that $u(x) > V(x)$. Suppose $M_\infty := L_\infty = \E[|x+\int_0^\infty e^{-rs}\,dW_s|]$.

Now consider the game defined via payoffs $L$, $U$ and $M_\infty$. Then the value function is given by $v^C=V$. Moreover the optimal strategy for both the \textsc{sup} and \textsc{inf} players is to never stop, so that $A^C=B^C =\emptyset$. 
Then Assumption \ref{assumption:technical} holds by the third element of Lemma~\ref{lem:sufftech} and the square integrability of the submartingale $L$.


\end{example}

We begin our study with a lemma which compares the expected payoff of the game under the independent constraint for two different stopping rules. 
The proof is given in the Appendix and relies on a coupling of marked Poisson processes.

\begin{lemma}\label{lem:comp1}
 Suppose $D$ and $E$ are disjoint {Borel} subsets of $\sX$. Fix $x\in \sX$.

 Suppose that $\tau\in \mathcal{R}^{M}(\lambda)$ satisfies $\mathbb{P}^x(\mbox{$\tau=\infty$ or $X_\tau\in D$})=1$. Then $(\tau\wedge \eta_E^{M},X_{\tau\wedge \eta_E^{M}})$ and $(\tau\wedge\eta_E^{m},X_{\tau\wedge\eta_E^{m}})$ have the same distribution and $J^x(\tau,\eta_E^{m})=J^x(\tau,\eta_E^{M})$.

 Similarly, given $\sigma\in \mathcal{R}^{m}(\lambda)$ such that $\mathbb{P}^x(\mbox{$\sigma=\infty$ or $X_\sigma\in E$})=1$, then $(\eta_D^{M}\wedge \sigma,X_{\eta_D^{M}\wedge \sigma})$ and $(\eta_D^{m}\wedge\sigma,X_{\eta_D^{m}\wedge\sigma})$ have the same distribution and  $J^x(\eta_D^{M},\sigma)=J^x(\eta_D^{m},\sigma)$.
\end{lemma}

Taking $\tau=\eta_{A^C}^{M}$ and $E=B^C$ in Lemma \ref{lem:comp1}, and using the fact that $J^x(\eta^C_D,\eta^C_E) = J^x(\eta^{M}_D,\eta^{M}_E) =  J^x(\eta^{m}_D,\eta^{m}_E)$ for any Borel sets $D$, $E$, we have:
\begin{corollary}\label{cor:commonequality}
    Suppose that the Dynkin game under the common Poisson constraint satisfies Assumption \ref{assumption:existence}, with $A^C\cap B^C=\emptyset$. Then
$\eta_{A^C}^C\wedge\eta_{B^C}^C$ and $\eta_{A^C}^{M}\wedge \eta_{B^C}^{m}$ are equally distributed and $v^C(x)=J^x(\eta^C_{A^C},\eta^C_{B^C})=J^x(\eta^{M}_{A^C},\eta^{m}_{B^C})$ for every $x\in \sX$.
\end{corollary}

Note that Assumption \ref{assumption:technical} is about the properties of $M_\infty$ on the set $\{\eta_{A^C}^C\wedge \eta_{B^C}^C=\infty\}$. Given that $\eta_{A^C}^C\wedge\eta_{B^C}^C$ and $\eta_{A^C}^{M}\wedge \eta_{B^C}^{m}$ have the same distribution, we may use Assumption~\ref{assumption:technical} to deduce that an analogous statement automatically holds in the independent constraint setting:
\begin{corollary}\label{cor:assumption}
    Suppose that Assumption \ref{assumption:technical} holds. Then, for every $x\in \sX$, the random variable $M_\infty $ satisfies $e^{-r\gamma}\E^{X^x_\gamma}[M_\infty\1_{\eta_{A^C}^{M}\wedge \eta_{B^C}^{m}=\infty}]=\E^x[M_\infty\1_{\eta_{A^C}^{M}\wedge \eta_{B^C}^{m}=\infty}|\mathcal{F}_{\gamma}]$, for any $\gamma \in \mathcal{R}^{M}(\lambda)$ or $\mathcal{R}^{m}(\lambda)$.
\end{corollary}


For a Borel set $D$ and for any $\tau \in \mathcal{R}^{M}(\lambda)$ we define $\gamma_{\tau,D}^{M}$ to be the first event time in $({T^{M}_n})_{n \geq 1}$ from $\tau$ onwards such that the diffusion process $X$ is in the set $D$: $\gamma_{\tau,D}^{M}=\inf\{t\geq \tau:t\in (T_j^{M})_{j\geq 1},X_t\in D\}$. Note that $\gamma_{\tau,D}^{M} \in \mathcal{R}^{M}(\lambda)$. Similarly, for any $\sigma\in \mathcal{R}^{m}(\lambda)$, define $\gamma_{\sigma,D}^{m}=\inf\{t\geq \sigma:t\in (T_j^{m})_{j\geq 1},X_t\in D\}\in \mathcal{R}^{m}(\lambda)$.

\begin{lemma}\label{lem:SMPtau}
Suppose that the Dynkin game under the  common Poisson constraint satisfies Assumption \ref{assumption:existence} and Assumption \ref{assumption:technical}. Fix $x\in \sX$. Then, for any $\tau \in \mathcal{R}^{M}(\lambda)$ and any $\sigma\in \mathcal{R}^{m}(\lambda)$, the following equalities hold:
\begin{eqnarray*}
     e^{-r\tau}v^C(X^x_\tau)\1_{\{\tau<\eta_{B^C}^{m}\}\cap\{\tau< \gamma_{\tau,A^C}^{M}\}}&=&\E^x[R(\gamma^{M}_{\tau,A^C},\eta_{B^C}^{m})|\mathcal{F}_\tau]\1_{\{\tau<\eta_{B^C}^{m}\}\cap\{\tau< \gamma_{\tau,A^C}^{M}\}};\\
    e^{-r\sigma}v^C(X^x_\sigma)\1_{\{\sigma<\eta_{A^C}^{M}\}\cap\{\sigma< \gamma_{\sigma,B^C}^{m}\}}&=&\E^x[R(\eta_{A^C}^{M},\gamma^{m}_{\sigma,B^C})|\mathcal{F}_\sigma]\1_{\{\sigma<\eta_{A^C}^{M}\}\cap\{\sigma< \gamma_{\sigma,B^C}^{m}\}}.
\end{eqnarray*}
\end{lemma}
\begin{proof}The proof follows by strong Markov property and is given in the Appendix.\end{proof}

\begin{proposition}\label{prop:main}
    {\bDGH{Suppose $M= L$.}}
    Suppose that the Dynkin game under the common Poisson constraint satisfies Assumption \ref{assumption:existence} and Assumption \ref{assumption:technical} and let the solution be given by $(v^C, A^C, B^C)$.

    Suppose further that the value function is such that $v^C \wedge l \leq u$.

    Then, the Dynkin game under the independent Poisson constraint has the same solution as the game under the common Poisson constraint. That is, $v^I(x)=v^C(x)$ for every $x\in \sX$, and moreover, $(\eta^{M}_{A^C},\eta^{m}_{B^C})$ is a saddle point of the game under the independent Poisson constraint.
\end{proposition}

\begin{proof}
Write $A=A^C$ and $B=B^C$.
Fix any $x\in \sX$. We aim to prove that $J^x(\tau,\eta_B^{m})\leq J^x(\eta_A^{M},\eta_B^{m})$ and $J^x(\eta_A^{M},\eta_B^{m})\leq J^x(\eta_A^{M},\sigma)$ for any $\tau\in \mathcal{R}^{M}(\lambda)$ and $\sigma \in \mathcal{R}^{m}(\lambda)$.  Note that
if $v^C \wedge l \leq u$ then \bDGH{$\{ x : v^C(x)\geq l(x)> u(x)\}=\emptyset$} so that $A=\{x:v^C(x)< l(x)\}$.

Step 1: Fix any $\tau \in \mathcal{R}^{M}(\lambda)$. We prove that $J^x(\tau,\eta_B^{m})\leq J^x(\gamma_{\tau,A}^{M},\eta_B^{m})$. We have:
   \begin{eqnarray*}
J^x(\tau,\eta_B^{m})&=&\E^x[L_\tau\1_{\{\tau<\eta_B^{m}\}}+U_{\eta_B^{m}}\1_{\{\eta_B^{m}<\tau\}}+M_\infty\1_{\{\tau=\eta_B^{m}=\infty\}}]\\
      &=& \E^x[e^{-r\tau}l(X_\tau)\1_{\{\tau<\eta_B^{m}\}}\1_{\{X_\tau\notin A\}}+L_{\tau}\1_{\{\tau<\eta_B^{m}\}}\1_{\{X_\tau \in A\}}+U_{\eta_B^{m}}\1_{\{\eta_B^{m}<\tau\}}\\&&\hspace{77mm}+M_\infty\1_{\{\tau=\eta_B^{m}=\infty\}}] \\
      &\bDGH{\leq} &\E^x[e^{-r\tau}v^C(X_\tau)\1_{\{\tau<\eta_B^{m}\}}\1_{\{X_\tau \notin A\}}+L_\tau\1_{\{\tau<\eta_B^{m}\}}\1_{\{X_\tau\in A\}}+U_{\eta_B^{m}}\1_{\{\eta_B^{m}<\tau\}}\\&&\hspace{86mm}+M_\infty\1_{\{\tau=\eta_B^{m}=\infty\}}] \\
      & = & \E^x[(L_{\gamma_{\tau,A}^{M}}\1_{\{\gamma_{\tau,A}^{M}<\eta_B^{m}\}}+U_{\eta_B^{m}}\1_{\{\gamma_{\tau,A}^{M}>\eta_B^{m}\}}+M_\infty\1_{\{\gamma_{\tau,A}^{M}=\eta_B^{m}=\infty\}})\1_{\{\tau<\eta_B^{m}\}}\1_{\{X_\tau \notin A\}}\\&&\hspace{21mm}+L_\tau\1_{\{\tau<\eta_B^{m}\}}\1_{\{X_\tau\in A\}}+U_{\eta_B^{m}}\1_{\{\eta_B^{m}<\tau\}}+M_\infty\1_{\{\tau=\eta_B^{m}=\infty\}}]
\end{eqnarray*}
where the inequality is by definition of $A$, and the subsequent equality is by Lemma \ref{lem:SMPtau}. Then, using that $\tau\leq \gamma_{\tau,A}^{M}$ and that $\Prob( \tau = \eta_B^{m} < \infty)=0$, we find
\begin{eqnarray*}
\lefteqn{J^x(\tau,\eta_B^{m})}\\      
    & \leq & \E^x[L_{\gamma_{\tau,A}^{M}}\1_{\{\gamma_{\tau,A}^{M}=\tau<\eta_B^{m}\}\cup \{\gamma_{\tau,A}^{M}<\eta_B^{m},\tau\neq \gamma_{\tau,A}^{M}\}}\\&&\hspace{15mm}+U_{\eta_B^{m}}\1_{\{\tau<\eta_B^{m}<\gamma_{\tau,A}^{M}\}\cup \{\eta_B^{m}<\tau\leq \gamma_{\tau,A}^{M}\}}+M_\infty\1_{\{\tau\leq \gamma_{\tau,A}^{M}=\eta_B^{m}=\infty\}}]\\
    &=& \E^x[L_{\gamma_{\tau,A}^{M}}\1_{\{\gamma_{\tau,A}^{M}<\eta_B^{m}\}}+U_{\eta_B^{m}}\1_{\{\eta_B^{m}< \gamma_{\tau,A}^{M}\}}+M_\infty\1_{\{\gamma_{\tau,A}^{M}=\eta_B^{m}=\infty\}}]\\&=&J^x(\gamma_{\tau,A}^{M},\eta_B^{m}).
\end{eqnarray*}

Step 2: Following step 1, any $\tau \in \mathcal{R}^{M}(\lambda)$ is no better than $\gamma_{\tau,A}^{M}$, therefore it is sufficient to prove that $\eta_A^{M}$ is at least as good a strategy as $\gamma_{\tau,A}^{M}$, for any $\tau$.

Note that $\gamma_{\tau,A}^{M}$ satisfies $\mathbb{P}^x(\gamma_{\tau,A}^{M}=\infty\,\textrm{or}\,X_{\gamma_{\tau,A}^{M}}\in A)=1$. Define $\mathcal{R}^{M}_A(\lambda)=\{\tau\in \mathcal{R}^{M}(\lambda):\,\mathbb{P}^x(\gamma_{\tau,A}^{M}=\infty\,\textrm{or}\,X_{\gamma_{\tau,A}^{M}}\in A)=1\}$, so for any $\tau \in \mathcal{R}^{M}(\lambda)$ the corresponding $\gamma_{\tau,A}^{M}\in \mathcal{R}^{M}_A(\lambda)$. It suffices to prove that $\sup_{\tau \in \mathcal{R}^{M}_A(\lambda)}J^x(\tau,\eta_B^{m})=J^x(\eta_A^{M},\eta_B^{m})$ holds.
We have:
\begin{eqnarray*}
    \sup_{\tau \in \mathcal{R}^{M}_A(\lambda)}J^x(\tau,\eta_B^{m})&=& \sup_{\tau \in \mathcal{R}^{M}_A(\lambda)}J^x(\tau,\eta_B^{M})\\
    &\leq & \sup_{\tau \in \mathcal{R}^{M}(\lambda)}J^x(\tau,\eta_B^{M})\\
    &=& J^x(\eta_A^{M},\eta_B^{M})\\
    &=& J^x(\eta_A^{M},\eta_B^{m}),
\end{eqnarray*}
where the first equality holds by Lemma \ref{lem:comp1}, the inequality holds as the supremum is taken over a larger set, the second equality holds by Assumption \ref{assumption:existence}, and the final equality is by Corollary \ref{cor:commonequality}.

It is clear that $\eta_A^{M}\in \mathcal{R}^{M}_A(\lambda)$, so we also have $\sup_{\tau \in \mathcal{R}^{M}_A(\lambda)}J^x(\tau,\eta_B^{m})\geq J^x(\eta_A^{M},\eta_B^{m})$. Therefore $\sup_{\tau \in \mathcal{R}^{M}_A(\lambda)}J^x(\tau,\eta_B^{m})=J^x(\eta_A^{M},\eta_B^{m})$ holds and we conclude.

The proof of $J^x(\eta^{M}_A,\eta_B^{m})\leq J^x(\eta_A^{M},\sigma)$ follows by a similar approach.

Hence, $(\eta_A^{M},\eta_B^{m})$ is a saddle point for the game under the independent constraint when $X_0=x$, with game value $v^I(x)=v^C(x)$ by Corollary \ref{cor:commonequality}. Given that the above argument holds for any $x\in \sX$ we can conclude that these properties hold for any $x\in \sX$.
\end{proof}

\bDGH{
\begin{lemma}
\label{lem:open}
Suppose that $v^C$ is continuous, $l$ is lower semicontinuous and $u$ is upper semicontinuous. Then $A^C$ and $B^C$ are both open.
\end{lemma}
\begin{proof}
Write $A^C = A^C_1 \cup A^C_2$ where $A^C_1 = \{x : v^C < l(x) \}$ and $A^C_2 = \{ x : v^C(x) \geq l(x) > u(x) \}$. 

{\bDGH{Suppose $f$ is lower semicontinuous and $g$ is upper semicontinuous.}}
Then the set $\{x : g(x)<f(x) \}$ is open. To see this note that if $z_n \in \{z: f(z) \leq g(z) \}$ and if $z_n \rightarrow z$ then $f(z) \leq \lim_n f(z_n) \leq \lim_n g(z_n) \leq g(z)$ so that $\{z : f(z) \leq g(z) \}$ is closed. Applying this twice and using that $v^C$ is both upper and lower semicontinuous, we get that $A^C_1$ and $B^C$ are open.

Now we want to show that $A^C$ is open.
If $x \in A^C_1$ then since $A^C_1$ is open there is a neighbourhood $N_1$ of $x$ such that $N_1 \subseteq A^C_1 \subseteq A^C$. Now suppose $x \in A^C_2 \setminus A^C_1$. Since $v^C(x)>u(x)$ and $l(x)>u(x)$ there exist neighbourhoods $N_v$ and $N_l$ of $x$ such that $v^C>u$ on $N_v$ and $l>u$ on $N_l$. Take $y \in N_v \cap N_l$. Either $v^C(y) \geq l(y)$ and $y \in A^C_2$ or $v^C(y)<l(y)$ and $y \in A^C_1$. It follows that $N_v \cap N_l \subseteq A^C$. Putting the two cases together we find that $A^C$ is open. 
\end{proof}

\begin{lemma}
\label{lem:sets}
$A^C \cap B^C= \emptyset$ if and only if $D^C := \{x : v^C \wedge l(x) \leq u(x) \} = \sX$.
\end{lemma}

\begin{proof}
Since $A_2^C \subseteq B^C$ we have that
$A^C \cap B^C = (A^C_1 \cup A^C_2) \cap B^C = (A^C_1 \cap B^C) \cup A^C_2$. But $A^C_1 \cap B^C= \{x: u(x)<v^C(x)< l(x) \}$ so that
$A^C \cap B^C = \{x: u(x)<v^C(x)< l(x) \} \cup \{x: u(x)<l(x) \leq v^C(x) \} = \{ u < v \wedge l \}$. In particular, $A^C \cap B^C =\emptyset$ if and only if $\{x : v^C(x) \wedge l(x) \leq u(x)\} = \sX$. 
\end{proof}
}


\begin{theorem}\label{thm:necessity}
Suppose $M= L$. Suppose that $l$ is lower semicontinuous and $u$ is upper semicontinuous.

Suppose that the Dynkin game under the common Poisson constraint satisfies Assumption \ref{assumption:existence} and Assumption \ref{assumption:technical} and let the solution be given by $(v^C, A^C, B^C)$ where 
$A^C=\{x:v^C(x)< l(x) \}\cup\{x:v^C(x) \geq l(x)> u(x)\}$, $B^C=\{x:v^C(x)> u(x)\}$.
Suppose that $v^C$ is continuous.

Then the following are equivalent.
\begin{enumerate}
    \item $v^C \wedge l \leq u$.
    \item $A^C \cap B^C = \emptyset$.
    \item $J^x(\eta^{M}_{A^C},\eta_{B^C}^{m})=J^x(\eta^C_{A^C},\eta^C_{B^C})$ for every $x\in \sX$.
\end{enumerate}
Moreover, if any of the above conditions hold, then the Dynkin game under the independent Poisson constraint has the same solution as the game under the common Poisson constraint. That is, $v^I(x)=v^C(x)$ for every $x\in \sX$, and $(\eta^{M}_{A^C},\eta^{m}_{B^C})$ is a saddle point of the game under the independent Poisson constraint.
\end{theorem}
\begin{proof}
    Write $A=A^C$ and $B=B^C$. The equivalence of 1. and 2. follows from Lemma~\ref{lem:sets}. The fact that 1. implies 3. and the final conclusion is Proposition~\ref{prop:main}. So, it only remains to prove that 3. implies 2. This follows by contradiction: we show that
    if $A\cap B$ is nonempty then $J^x(\eta^{M}_{A},\eta_{B}^{m})<J^x(\eta^C_{A},\eta^C_{B})$. 

    Suppose that $A\cap B$ is nonempty, and fix arbitrary $x\in \sX$. By Lemma \ref{lem:open} $A\cap B$ is open hence must have a positive Lebesgue measure. By Lemma \ref{lem:comp1} we have $v^C(x)=J^x(\eta^{C}_{A},\eta_{B }^{C})=J^x(\eta^{C}_{A},\eta_{B\setminus A }^{C})={J}^x(\eta^{M}_{A},\eta_{B\setminus A }^{m})$, so to seek a contradiction it suffices to prove that $J^x(\eta^{M}_{A},\eta_{B}^{m})<{J}^x(\eta^{M}_{A},\eta_{B\setminus A }^{m})$.    

    Consider a modified game with payoffs $\tilde{l}=l$ and $\tilde{u}(x)=v^C(x)\1_{x\in A\cap B}+u(x)\1_{x\notin A\cap B}$, and denote the expected payoff of this modified game by $\tilde{J}$. By construction we have $\tilde{u} \geq u$ and hence  $J^x(\eta^{M}_{A},\eta_{B}^{m})<\tilde{J}^x(\eta^{M}_{A},\eta_{B}^{m})$.
    
    We will prove that, under the assumption that $A\cap B$ is nonempty, ${J}^x(\eta^{M}_{A},\eta_{B\setminus A}^{m})=\tilde{J}^x(\eta^{M}_{A},\eta_{B}^{m})$. We have
    \begin{eqnarray*}
        \lefteqn{\tilde{J}^x(\eta^{M}_{A},\eta_{B}^{m})}\\&=&\E^x[{R}(\eta^{M}_{A},\eta_{B}^{m})\1_{\{X_{\eta^{M}_{A}\wedge \eta_{B}^{m}}\notin A\cap B\}\cup \{\eta_{A}^M<\eta_{B}^m\}}\\&&\hspace{24mm}+e^{-r\eta_{B}^m}v^C(X_{\eta_{B}^m})\1_{\{X_{\eta_{B}^{m}}\in A\cap B\}\cap \{\eta_{A}^M>\eta_{B}^m\}}+M_\infty\1_{\eta_{A}^M=\eta_{B}^m=\infty}]\\ 
        &=&\E^x[{R}(\eta^{M}_{A},\eta_{B\setminus A}^{m})\1_{\{X_{\eta^{M}_{A}\wedge \eta_{B}^{m}}\notin A\cap B\}\cup \{\eta_{A}^M<\eta_{B}^m\}}\\&&\hspace{24mm}+e^{-r\eta_{B}^m}v^C(X_{\eta_{B}^m})\1_{\{X_{\eta_{B}^{m}}\in A\cap B\}\cap \{\eta_{A}^M>\eta_{B}^m\}}+M_\infty\1_{\eta_{A}^M=\eta_{B}^m=\infty}]
 \end{eqnarray*}
where in the second equality we use the fact that $\eta_A^M\wedge \eta_{B\setminus A}^m=\eta_A^M\wedge \eta_B^m$ on the set $\{X_{\eta^{M}_{A}\wedge \eta_{B}^{m}}\notin A\cap B\}\cup \{\eta_{A}^M<\eta_{B}^m\}$,
Then, applying the strong Markov property and using $v^C(x)=J^x(\eta^{M}_{A},\eta_{B\setminus A }^{m})$ we find   
\begin{eqnarray*}
\lefteqn{\tilde{J}^x(\eta^{M}_{A},\eta_{B}^{m})}\\    
        &=&\E^x[R(\eta^{M}_{A},\eta_{B\setminus A}^{m})\1_{\{X_{\eta^{M}_{A}\wedge \eta_{B}^{m}}\notin A\cap B\}\cup \{\eta_{A}^M<\eta_{B}^m\}}\\
        &&+(L_{\eta_{A}^M}\1_{\eta_A^M<\eta_{B\setminus A}^m}+U_{\eta_{B\setminus A}^m}\1_{\eta_{B\setminus A}^m<\eta_{A}^M}+M_\infty\1_{\eta_{B\setminus A}^m=\eta_{A}^M=\infty})\1_{\{X_{\eta_{B}^{m}}\in A\cap B\}\cap \{\eta_{A}^M>\eta_{B}^m\}}\\&&\hspace{93mm}+M_\infty\1_{\eta_{A}^M=\eta_{B}^m=\infty}]\\
        &=&\E^x[{R}(\eta^{M}_{A},\eta_{B\setminus A}^{m})\1_{\{X_{\eta^{M}_{A}\wedge \eta_{B}^{m}}\notin A\cap B\}\cup \{\eta_{A}^M<\eta_{B}^m\}}+R(\eta^{M}_{A},\eta_{B\setminus A}^{m})\1_{\{X_{\eta_{B}^{m}}\in A\cap B\}\cap \{\eta_{A}^M>\eta_{B}^m\}}\\&&\hspace{90mm}+M_\infty\1_{\eta_{A}^M=\eta_{B\setminus A}^m=\infty}]\\
        &=& {J}^x(\eta^{M}_{A},\eta_{B\setminus A}^{m}).
    \end{eqnarray*}   

Hence we conclude that  $J^x(\eta^{M}_{A},\eta_{B}^{m})< \tilde{J}^x(\eta^{M}_{A},\eta_{B}^{m})=J^x(\eta^C_{A},\eta^C_{B})$ for any $x\in \sX$ and this leads to a contradiction. Therefore 3. implies 2. and the proof is complete.
\end{proof}

\begin{remark}\label{rem:necessity}
{
    The equivalence of 2. in the above theorem may not hold if we make different choices of $A^C$ or $B^C$, \bDGH{even if those different choices also define a saddlepoint}.} 
    For example, if we take $\tilde{A}^C=\{x:v^C(x)\leq  l(x) \}\cup\{x:v^C(x)>l(x)> u(x)\}$ instead, then it could be the case that there exists some $x$ such that $u(x)\leq v^C(x)\leq l(x)$ holds, so that $\tilde{A}^C\cap B^C\neq \emptyset$, but the game under the independent Poisson constraint still has the same value as the game under the common Poisson constraint. In particular, without some restrictions on the choices of stopping sets, the condition $A^C \cap B^C = \emptyset$ of Theorem~\ref{thm:necessity} is not necessary for the conclusion of Proposition~\ref{prop:main} to hold. 
\end{remark}

Theorem \ref{thm:necessity} tells us that, subject to some regularity conditions at infinity, if we have solved a Dynkin game under the common Poisson constraint and there exist optimal stopping regions for the two players which do not overlap, then the corresponding game under the independent Poisson constraint is also solved, with the same value and the same optimal stopping regions. The converse of this statement is not true. That is, if we have solved the game under the independent Poisson constraint and the optimal stopping regions are disjoint, then the corresponding game under common Poisson constraint may not have the same value. This is the subject of the next subsection.

\subsection{From the independent constraint to the common constraint}
\label{ssec:itoc}
Our goal now is to consider what happens if we begin with a solution of the game under the independent  constraint and ask for sufficient conditions for the solution to also be the solution of the game under the common constraint.
We begin by making a similar assumption to Assumption \ref{assumption:existence} in Section~\ref{ssec:ctoi}.


\begin{assumption}\label{assumption:existence1}
    The Dynkin game under the independent Poisson constraint has a value $\{v^I(x)\}_{x\in \sX}$, with saddle point $(\eta^{M}_{A^I},\eta^{m}_{B^I})$, where $A^I=\{x:v^I(x)< l(x) \}$ and \bDGH{$B^I=\{x:v^I(x)> u(x)\}$}.
\end{assumption}

Note that the optimal stopping set for the \textsc{sup} player has a different form to that in Assumption~\ref{assumption:existence}, whereas the optimal stopping set for the \textsc{inf} player has the same form. This asymmetry arises from our convention about what happens when both players stop simultaneously.

\begin{assumption}\label{assumption:technical1}
    For every $x\in \sX$, the random variable $M_\infty$ and stopping times $\eta_{A^I}^{M}$, $\eta_{B^I}^{m}$ satisfy $e^{-r\gamma}\E^{X^x_\gamma}[M_\infty\1_{\eta_{A^I}^{M}\wedge \eta_{B^I}^{m}=\infty}]=\E^x[M_\infty\1_{\eta_{A^I}^{M}\wedge \eta_{B^I}^{m}=\infty}|\mathcal{F}_{\gamma}]$, for any $\gamma \in \mathcal{R}^{M}(\lambda)$ or $\mathcal{R}^{m}(\lambda)$, where $A^I$ and $B^I$ are as defined in Assumption \ref{assumption:existence1}.
\end{assumption}

By Lemma \ref{lem:comp1}, we have the following analogue of Corollary~\ref{cor:commonequality}:
\begin{corollary}\label{cor:indepequality}
    {Suppose that the Dynkin game under the independent Poisson constraint satisfies Assumption \ref{assumption:existence1}, with $A^I\cap B^I=\emptyset$.} Then
$\eta_{A^I}^C\wedge\eta_{B^I}^C$ and $\eta_{A^I}^{M}\wedge \eta_{B^I}^{m}$ are equally distributed and $v^I(x)=J^x(\eta^{M}_{A^I},\eta^{m}_{B^I})=J^x(\eta^C_{A^I},\eta^C_{B^I})$ for every $x\in \sX$.
\end{corollary}

\bDGH{
\begin{lemma}
\label{lem:setsi}
$D^I := \{x : v^I \wedge l(x) \leq u(x) \} = \sX$ if and only if $A^I \cap \bDGH{B^I} = \emptyset$ and $\{x:v^I(x)\geq l(x)>u(x)\}=\emptyset$.
\end{lemma}}

\begin{proposition}\label{prop:main1}
    Suppose that the Dynkin game under the independent Poisson constraint satisfies Assumption \ref{assumption:existence1} and Assumption \ref{assumption:technical1}.

    Suppose that the value function is such that $v^I \wedge l \leq u$.


    Then, the Dynkin game under the common Poisson constraint has the same value as the game under the independent Poisson constraint. That is, $v^C(x)=v^I(x)$ for every $x\in \sX$. Moreover, $(\eta^{C}_{A^I},\eta^{C}_{B^I})$ is a saddle point of the game under the common Poisson constraint.
\end{proposition}

Recall that, in the proof of Proposition \ref{prop:main}, the first step is to prove that, given any strategy $\tau\in \mathcal{R}^{M}(\lambda)$, it is always better for the \textsc{sup} player to choose $\gamma_{\tau,A^I}^{M}$. Under the common constraint this is not always the case, because the \textsc{sup} player may benefit from preempting the \textsc{inf} player's stopping decision. For arbitrary $\tau\in \mathcal{R}^C(\lambda)$, consider the set $\{\tau=\eta_{B^I}^C<\infty\}$. It is clear that $X_\tau\in B^I$ and $\gamma_{\tau,A^I}^C>\tau$ holds on this set (assuming $A^I\cap B^I=\emptyset$), and the \textsc{sup} player can either stop at $\tau$ and get $L_\tau$, or stop later and get $U_\tau$. So $\gamma_{\tau,A^I}^C$ is better in this case only if $U_\tau\geq L_\tau$, which is guaranteed under the extra assumption that the set $\{x:v^I(x)\geq l(x)>u(x)\}$ is empty.
This explains why, although the statement of Proposition~\ref{prop:main1} is very similar to that of Proposition \ref{prop:main}, the subsequent corollaries are different.




As an example in the next section will show, the condition that $\{x:v^I(x) \geq l(x)>u(x)\}$ is empty is necessary, and the fact that $A^I$ and $B^I$ are disjoint alone is not sufficient for $(v^I, A^I,B^I)$ to be a solution of the common constraint problem.

\begin{proof}[Proof of Proposition~\ref{prop:main1}]
Fix any $x\in \sX$. We will prove that, given any strategy $\tau \in \mathcal{R}^C(\lambda)$, $\gamma_{\tau,A^I}^C$ is a better strategy for the \textsc{sup} player. For typographical reasons, for the rest of the proof we write $A$ and $B$ as shorthand for $A^I$ and $B^I$.

Fix any $\tau \in \mathcal{R}^{C}(\lambda)$. By a similar argument as in the proof of step 1 in Proposition \ref{prop:main}, we have:
    \begin{eqnarray*}
      \lefteqn{J^x(\tau,\eta_B^{C})}\\&=&\E^x[L_\tau\1_{\{\tau<\eta_B^{C}\}}+U_{\eta_B^{C}}\1_{\{\eta_B^{C}<\tau\}}+L_\tau\1_{\{\tau=\eta_B^C<\infty\}}+M_\infty\1_{\{\tau=\eta_B^C=\infty\}}]\\
     &\bDGH{\leq}&\E^x[e^{-r\tau}v^I(X_\tau)\1_{\{\tau<\eta_B^{C}\}}\1_{\{X_\tau \notin A\}}+L_\tau\1_{\{\tau<\eta_B^{C}\}}\1_{\{X_\tau\in A\}}+U_{\eta_B^{C}}\1_{\{\eta_B^{C}<\tau\}}\\&&\hspace{35mm}+U_{\eta_B^{C}}\1_{\{\tau=\eta_B^C<\infty\}}\1_{\{X_{\tau}\notin A\}}+M_\infty\1_{\{\tau=\eta_B^C=\infty\}}]\\
    &=& \E^x[L_{\gamma_{\tau,A}^{C}}\1_{\{\gamma_{\tau,A}^{C}\leq \eta_B^{C}, \gamma_{\tau,A}^C<\infty\}}+U_{\eta_B^{C}}\1_{\{\eta_B^{C}< \gamma_{\tau,A}^{C}\}}+M_\infty\1_{\{\gamma_{\tau,A}^C=\eta_B^C=\infty\}}]\\&=&J^x(\gamma_{\tau,A}^{C},\eta_B^{C}),
\end{eqnarray*}
where, for the inequality, the set $\{\tau=\eta_B^C<\infty\}\cap \{X_\tau\in A\}$ is empty, and $L_\tau\leq U_\tau$ holds on the set $\{\tau=\eta_B^C<\infty\}\cap \{X_\tau\notin A\}$ by the assumption that $v^I\wedge l\leq u$ and Lemma \ref{lem:setsi}.

The proof of $J^x(\eta_A^C,\gamma_{\sigma,B}^C)\leq J^x(\eta_A^C,\sigma)$ does not rely on the extra assumption that the set $\{x:v^I(x)\geq l(x)>u(x)\}$ is empty, hence follows a similar approach as in Step 1 of Proposition \ref{prop:main}. Then, by similar argument as in Step 2 of Proposition \ref{prop:main}, the saddle point property of $(\eta^{C}_{A},\eta^{C}_B)$ follows, and the value of the game under the common constraint is $v^C(x)=v^I(x)$ by Corollary \ref{cor:indepequality}. Hence we conclude.
\end{proof}

\begin{theorem}\label{thm:necessityI}
Suppose $M=L$. Suppose that $l$ is lower semicontinuous and $u$ is upper semicontinuous.

Suppose that the Dynkin game under the independent Poisson constraint satisfies Assumption \ref{assumption:existence1} and Assumption \ref{assumption:technical1} and let the solution be given by $(v^I, A^I, B^I)$ where 
$A^I=\{x:v^I(x)< l(x) \}$, $B^I=\{x:v^I(x)> u(x)\}$
Suppose that $v^I$ is continuous.

Then the following are equivalent.
\begin{enumerate}
    \item $v^I \wedge l \leq u$.
    \item $A^I \cap B^I = \emptyset$ and $\{x:v^I(x)\geq l(x)>u(x)\}=\emptyset$.
    \item $J^x(\eta^{C}_{A^I},\eta^{C}_{B^I})=J^x(\eta^{M}_{A^I},\eta^{m}_{B^I})$ for every $x\in \sX$.    
\end{enumerate}
Moreover, if any of the above conditions hold, then the Dynkin game under the common Poisson constraint has the same solution as the game under the independent Poisson constraint. That is, $v^C(x)=v^I(x)$ for every $x\in \sX$, and $(\eta^{C}_{A^I},\eta^{C}_{B^I})$ is a saddle point of the game under the common Poisson constraint.    
\end{theorem}
\begin{proof}
    The equivalence of 1. and 2. is by Lemma \ref{lem:setsi} and Proposition \ref{prop:main1} states that 1. implies 3. and the final conclusion. The proof that 3. implies 2. is similar to that in Theorem \ref{thm:necessity} and we seek a contradiction.  Again we consider a modified game defined by $\tilde{l}=l$ and $\tilde{u}(x)=v^I(x)\1_{x\in (A^I\cap B^I)\cup \{x:v^I(x)\geq l(x)>u(x)\}}+u(x)\1_{x\notin (A^I\cap B^I)\cup \{x:v^I(x)\geq l(x)>u(x)\}}$. Then, by strong Markov property, the assumption that 2. fails implies that  $J^x(\eta^{C}_{A^I},\eta^{C}_{B^I})>J^x(\eta^{M}_{A^I},\eta^{m}_{B^I})$.
    \end{proof}

A sufficient condition for the set $\{x:v^I(x)\bDGH{\geq}l(x)>u(x)\}$ to be empty is $l(x)\leq u(x)$, under which the Dynkin games under the common and the independent constraints are equivalent. However, this order condition is not necessary for $\{x:v^I(x)\geq l(x)>u(x)\}=\emptyset$. We will see an example in the next section where the order condition fails, but the hypotheses of Theorem~\ref{thm:necessityI} are satisfied.

\begin{corollary}
Suppose that the order condition $l \leq u$ holds. 

Suppose that Assumptions~\ref{assumption:existence} and \ref{assumption:technical} hold.
If $(v^C, A^C, B^C)$ is a solution of the game under the common Poisson constraint then $(v^C, A^C, B^C)$ is also the solution of the game under the independent Poisson constraint.

Suppose that Assumptions~\ref{assumption:existence1} and \ref{assumption:technical1} hold.  
If $(v^I, A^I, B^I)$ is a solution of the game under the independent Poisson constraint then $(v^I, A^I, B^I)$ is also the solution of the game under the common Poisson constraint.
\end{corollary}

\subsection{Necessity of the sufficient condition}\label{ssec:necessity}

In the previous two subsections we introduced sufficient conditions under which the solution to the Dynkin game under one constraint is also the solution under the other constraint. Under the following extra assumption the sufficient conditions are also necessary:
\begin{assumption}\label{assumption:gamevalue}
    (i) The value $v^C(x)$ of the game under the common Poisson constraint exists for every $x\in \sX$ and satisfies:
    \[v^C(x)=\E^x[e^{-rT_1^C}\max\{l(X_{T_{1}^C}),\min(v^C(X_{T_{1}^C}),u(X_{T_{1}^C}))\}].\]

    (ii) The value $v^I(x)$ of the game under the independent Poisson constraint exists for every $x\in \sX$ and satisfies
    \[v^I(x)=\E^x[e^{-r(T_1^M\wedge T_1^m)}(\max\{l(X_{T_{1}^M}),v^{\bDGH{I}}(X_{T_{1}^M})\}\1_{T_1^M<T_1^m}+\min\{u(X_{T_1^m}),v^I(X_{T_1^m})\}\1_{T_1^m<T_1^M})].\]
\end{assumption}

We will show in Section~\ref{sec:BSDE} (see Theorems~\ref{thm:satisfyassC} and \ref{thm:satisfyassi}) that Assumption~\ref{assumption:gamevalue} is satisfied in a wide class of examples.

\begin{proposition}
    Suppose that $M=L$, $l$ is lower semi-continuous and $u$ is upper semi-continuous. Suppose that Assumption \ref{assumption:gamevalue} holds.

    Suppose that $v^I=v^C$ and that $v$ is continuous where $v:= v^I \equiv v^C$. Then $v\wedge l\leq u$.
\end{proposition}

\begin{proof}
     To seek a contradiction, suppose that $v^I=v^C$ and $\{x: u(x)<v(x)\wedge l(x)\}$ is nonempty. By the argument in Lemma \ref{lem:open}, the set $\{x: u(x)<v(x)\wedge l(x)\}$ is open, hence it must have a positive measure.

    Define the following subsets of the state space $\sX$:
    \begin{eqnarray*}
        E_1=\{x:v(x)<l(x)\leq u(x)\}; & \hspace{15mm} &E_2=\{x:l(x)\leq v(x)\leq u(x)\};\\
        E_3=\{x:l(x)\leq u(x)<v(x)\}; && E_4=\{x:v(x)\leq u(x)<l(x)\};\\
        E_5=\{x:u(x)<v(x)<l(x)\}; &&E_6=\{x:u(x)<l(x)\leq v(x)\}.
    \end{eqnarray*}
    The sets $E_i$, $i=1,...,6$ form a partition of $\sX$, with $\{x: u(x)<v(x)\wedge l(x)\}=E_5\cup E_6$. 

    Suppose that the event times $\{T_j^M\}_{j \geq 1},\{T_k^m\}_{k \geq 1}$ are generated by two independent Poisson processes, each of rate $\lambda$. Following a similar coupling argument as that in the proof of Lemma \ref{lem:comp1}, we can thin out all the event times or signals from the \textsc{sup} player's Poisson process when the process $X$ is in $E_3\cup E_6$ and thin out all the signals from the \textsc{inf} player's Poisson process when the process $X$ is in $E_1\cup E_2\cup E_4 \cup E_5$. \bDGH{Combining the remaining event times yields} a Poisson process \bDGH{of rate $\lambda$} and we denote the event times by $\{Q_n\}_{n\geq 1}$. Alternatively we can generate another Poisson process with intensity $\lambda$ by thinning out all the signals from the \textsc{inf} player's Poisson process when the process $X$ is in $E_3\cup E_6$ and all the signals from the \textsc{sup} player's Poisson process when the process $X$ is in $E_1\cup E_2\cup E_4 \cup E_5$. We denote the event times of this new Poisson process by $\{S_n\}_{n\geq 1}$. Note that $\{Q_m\}_{m\geq 1} \cup \{S_n\}_{n\geq 1}= \{T_j^M\}_{j \geq 1} \cup \{T_k^m\}_{k \geq 1}$.

    By Assumption \ref{assumption:gamevalue}(ii) and the definition of $E_i$, 
{\footnotesize{
    \begin{eqnarray*}
    v^I(x)
        &=& \E^x\bigg[e^{-r(T_1^M\wedge T_1^m)}\bigg\{\1_{X_{T_1^M\wedge T_1^m}\in E_1}(l(X_{T_{1}^M})\1_{T_1^M<T_1^m}+v^I(X_{T_1^m})\1_{T_1^m<T_1^M})\\
        &&\hspace{25mm}+\1_{X_{T_1^M\wedge T_1^m}\in E_2}(v^I(X_{T_{1}^M})\1_{T_1^M<T_1^m}+v^I(X_{T_1^m})\1_{T_1^m<T_1^M})\\
        &&\hspace{25mm}+\1_{X_{T_1^M\wedge T_1^m}\in E_3}(v^I(X_{T_{1}^M})\1_{T_1^M<T_1^m}+u(X_{T_1^m})\1_{T_1^m<T_1^M})\\
        &&\hspace{25mm}+\1_{X_{T_1^M\wedge T_1^m}\in E_4}(l(X_{T_{1}^M})\1_{T_1^M<T_1^m}+v^I(X_{T_1^m})\1_{T_1^m<T_1^M})\\
        &&\hspace{25mm}+\1_{X_{T_1^M\wedge T_1^m}\in E_5}(l(X_{T_{1}^M})\1_{T_1^M<T_1^m}+u(X_{T_1^m})\1_{T_1^m<T_1^M}) \\
        &&\hspace{25mm}+\1_{X_{T_1^M\wedge T_1^m}\in E_6}(v^I(X_{T_{1}^M})\1_{T_1^M<T_1^m}+u(X_{T_1^m})\1_{T_1^m<T_1^M})\bigg\}\bigg] \\
       &< & \E^x\bigg[e^{-r(T_1^M\wedge T_1^m)}\bigg\{\1_{X_{T_1^M\wedge T_1^m}\in E_1}(l(X_{T_{1}^M})\1_{T_1^M<T_1^m}+v^I(X_{T_1^m})\1_{T_1^m<T_1^M})\\
        &&\hspace{25mm}+\1_{X_{T_1^M\wedge T_1^m}\in E_2}(v^I(X_{T_{1}^M})\1_{T_1^M<T_1^m}+v^I(X_{T_1^m})\1_{T_1^m<T_1^M})\\
        &&\hspace{25mm}+\1_{X_{T_1^M\wedge T_1^m}\in E_3}(v^I(X_{T_{1}^M})\1_{T_1^M<T_1^m}+u(X_{T_1^m})\1_{T_1^m<T_1^M})\\
        &&\hspace{25mm}+\1_{X_{T_1^M\wedge T_1^m}\in E_4}(l(X_{T_{1}^M})\1_{T_1^M<T_1^m}+v^I(X_{T_1^m})\1_{T_1^m<T_1^M})\\
        &&\hspace{25mm}+\1_{X_{T_1^M\wedge T_1^m}\in E_5}(l(X_{T_{1}^M})\1_{T_1^M<T_1^m}+v^I(X_{T_1^m})\1_{T_1^m<T_1^M})\\
        &&\hspace{25mm}+\1_{X_{T_1^M\wedge T_1^m}\in E_6}(v^I(X_{T_{1}^M})\1_{T_1^M<T_1^m}+l(X_{T_1^m})\1_{T_1^m<T_1^M})\bigg\}\bigg] 
\end{eqnarray*} }}
where the strict inequality follows from the fact that $E_5 \cup E_6$ has positive measure.
Rewriting the last expression in terms of $Q_1$ and $S_1$, and using $v^I \equiv v^C$, the strong Markov property for $v^C(X_{S_1})$ and Assumption~\ref{assumption:gamevalue}(i) we have
{\footnotesize{
\begin{eqnarray*}
v^I(x)        &<& \E^x\bigg[e^{-r(Q_1\wedge S_1)}\bigg\{\1_{X_{Q_1\wedge S_1}\in E_1}(l(X_{Q_1})\1_{Q_1<S_1}+v^I(X_{S_1})\1_{S_1<Q_1})\\
        &&\hspace{25mm}+\1_{X_{Q_1\wedge S_1}\in E_2}(v^I(X_{Q_1})\1_{Q_1<S_1}+v^I(X_{S_1})\1_{S_1<Q_1})\\
        &&\hspace{25mm}+\1_{X_{Q_1\wedge S_1}\in E_3}(v^I(X_{S_1})\1_{S_1<Q_1}+u(X_{Q_1})\1_{Q_1<S_1})\\
        &&\hspace{25mm}+\1_{X_{Q_1\wedge S_1}\in E_4}(l(X_{Q_1})\1_{Q_1<S_1}+v^I(X_{S_1})\1_{S_1<Q_1})\\
        &&\hspace{25mm}+\1_{X_{Q_1\wedge S_1}\in E_5}(l(X_{Q_1})\1_{Q_1<S_1}+v^I(X_{S_1})\1_{S_1<Q_1})\\
        &&\hspace{25mm}+\1_{X_{Q_1\wedge S_1}\in E_6}(v^I(X_{S_1})\1_{S_1<Q_1}+l(X_{Q_1})\1_{Q_1<S_1})\bigg\}\bigg]\\
        &=&\E^x[e^{-r S_1}v^C(X_{S_1})\1_{S_1<Q_1}+e^{-rQ_1}\max\{l(X_{Q_1}),\min(v^C(X_{Q_1}),u(X_{Q_1}))\}\1_{Q_1<S_1}]\\
        &=& \E^x[e^{-r Q_1}\max\{l(X_{Q_1}),\min(v^C(X_{Q_1}),u(X_{Q_1}))\}\1_{S_1<Q_1}\\&&\hspace{40mm}+e^{-rQ_1}\max\{l(X_{Q_1}),\min(v^C(X_{Q_1}),u(X_{Q_1}))\}\1_{Q_1<S_1}]\\
        &=&v^C(x).
    \end{eqnarray*}
}} But then we have $v^I(x) < V = v^C(x)$ which contradicts with our assumption that $v^C=v^I$.
\end{proof}

\section{An example and a counterexample}
\label{sec:examples}

\subsection{From the common constraint to the independent constraint: an example}

In this subsection, we construct an example where the players' optimal stopping regions are disjoint under the common constraint. As a consequence, Theorem~\ref{thm:necessity} tells us the solution (value function and optimal stopping regions of the game) of the game under the common constraint is also a solution under the independent constraint.

Let $X$ be a Brownian motion. Consider the game with payoff functions: \bDGH{$l(x)=\1_{x\in(-1,1)}$} and $u(x)=\frac{\lambda}{\lambda+r}\1_{x\in[-1,1]}+\frac{1}{1+\epsilon}\1_{x\notin [-1,1]}$, where $\epsilon>0$ is sufficiently large (see \eqref{eq:epsbound} below). \bDGH{Note that we have chosen $l$ and $u$ to be lower-continuous and upper-continuous respectively.}
We also take $M_\infty=0$. (Note that $\lim\limits_{t \uparrow \infty} e^{-rt} l(X_t) = 0 = \lim\limits_{t \uparrow \infty} e^{-rt} u(X_t)$ so this is the natural candidate for the payoff if neither player stops.)

These functions define a Dynkin game with payoff:
\begin{equation} \label{eq:example}
    R(\tau,\sigma)=e^{-r\tau}\bDGH{\1_{X_\tau\in(-1,1)}}\1_{\tau\leq \sigma} \1_{\tau < \infty} +e^{-r\sigma}\left(\frac{\lambda}{\lambda+r}\1_{X_\sigma\in[-1,1]}+\frac{1}{1+\epsilon}\1_{X_\sigma\notin [-1,1]}\right)\1_{\sigma<\tau}.
\end{equation}
We solve the game with expected payoff $J^x(\tau,\sigma) = \E^x[R(\tau,\sigma)]$ under the common constraint. To simplify the expressions, we define $\theta=\sqrt{2r}$ and $\phi=\sqrt{2(\lambda+r)}$ and use these notations instead of $\lambda$ and $r$ when we present the value of the games. Also, note that the functions $l$, $u$ and the value are all even functions and our derivation will focus on the positive real line with the values on the negative real line following by symmetry. 

We assume that
\begin{equation}\label{eq:epsbound}
    \epsilon>\frac{\theta^2\sinh(\phi)+\theta\phi\cosh(\phi)}{(\phi^2-\theta^2)\sinh(\phi)}.
\end{equation}
Standard arguments imply that the value of this game should satisfy the HJB equation
\begin{equation}\label{HJB1}
    \mathcal{L}V-\frac{\phi^2}{2} V +\frac{\phi^2-\theta^2}{2}\max\{\min(V,u),l\}=0,
\end{equation}
where $\mathcal{L}f=\frac{1}{2}f''$ is the infinitesimal generator of a Brownian motion. Therefore our strategy is to construct a solution to \eqref{HJB1} and then to verify that this solution is the value of the game.

We expect that the \textsc{sup} player seeks to stop when $|X|$ is small whereas the \textsc{inf} player seeks to stop when $|X|$ is moderate but not too large (when $|X|$ is large the \textsc{inf} player hopes that discounting will reduce the value of any payoff to the \textsc{sup} player over time, and does not wish to stop). Since $l$ has a jump at $\pm 1$ we expect that the threshold between small and moderate is at $\pm 1$. We expect the threshold between moderate and large to occur at some point $x^*$ which is characterised by $V(x^*)= \frac{1}{1+\epsilon}$.
We further expect that the value function of the game will be decreasing on the positive real line.

Define the function $f:(0,\infty) \rightarrow (0,\infty)$ by $f(x)=\theta( \theta \sinh(\phi x)+ \phi \cosh(\phi x))$. Recall the lower bound \eqref{eq:epsbound} on $\epsilon$.

\begin{lemma}\label{lem:xstar}
    The equation $f(x)=(\phi^2-\theta^2)\epsilon\sinh(\phi)$ has a unique solution, which we label $\hat{x}$. Furthermore, $\hat{x} \in (1,\infty)$.
\end{lemma}
\begin{proof}
Observe that $f'(x)=\theta \phi (\theta \cosh(\phi x)+\phi \sinh(\phi x))>0$ on $[0,\infty)$, and that $\lim\limits_{x \uparrow \infty} f(x) = \infty$. Therefore $f$ is increasing on $(0,\infty)$ and any solution of $f(x) = f_0$ is unique. The assumption that $\epsilon>\frac{\theta^2\sinh(\phi)+\theta\phi\cosh(\phi)}{(\phi^2-\theta^2)\sinh(\phi)}$ implies that $f(1)<(\phi^2-\theta^2)\epsilon\sinh(\phi)$. It follows that there exists a unique positive solution $\hat{x}$ to $f(\hat{x})=(\phi^2-\theta^2)\epsilon\sinh(\phi)$ and $\hat{x}>1$.
\end{proof}

Define the candidate value function
\begin{equation}\label{eq:value}
V(x) = \left\{ \begin{array}{lll}
De^{\theta x} & \; & x<-x^*;\\
Be^{-\phi x}+Ce^{\phi x}+\frac{\phi^2-\theta^2}{\phi^2}\frac{1}{1+\epsilon} & \; & x\in[-x^*,-1);\\
   A\cosh(\phi x)+\frac{\phi^2-\theta^2}{\phi^2} & \; & x\in[-1,1];\\
Be^{\phi x}+Ce^{-\phi x}+\frac{\phi^2-\theta^2}{\phi^2}\frac{1}{1+\epsilon} & \; & x\in(1,x^*];\\
De^{-\theta x} & \; & x>x^*.
\end{array} \right.
\end{equation}
where $A$, $B$, $C$, $D$ and $x^*$ are to be determined.

 {Recall that a strong solution of a HJB equation is
a twice (weakly) differentiable function satisfying the equation almost everywhere (See Gilbarg and Trudinger \cite{Gilbarg2001} Chapter 9).}

\begin{lemma}\label{lem:propertyV} Set $x^*= \hat{x}$ where $\hat{x}$ is as defined in Lemma \ref{lem:xstar} and set
\begin{eqnarray*}
    A&=&\frac{\theta^2-\theta \phi}{\phi^2}\frac{1}{1+\epsilon} e^{-\phi x^*}-\frac{\phi^2-\theta^2}{\phi^2}\frac{\epsilon}{1+\epsilon}e^{-\phi}<0;\\
    B&=&\frac{\theta^2-\theta \phi}{2\phi^2}\frac{1}{1+\epsilon} e^{-\phi x^*}<0;   \\
    C&=&\frac{\theta^2+\theta\phi}{2\phi^2}\frac{1}{1+\epsilon} e^{\phi x^*}>0;\\
    D&=&\frac{1}{1+\epsilon} e^{\theta x^*}>0.
\end{eqnarray*}

    Then the function $V$ is nonnegative, even, decreasing on $[0,\infty)$, and is in $C^1$. The first derivative $V'$ is bounded, and the second derivative $V''$ exists and is continuous except at $x=\pm 1$. The second left and right derivatives of $V$ exist at $x=\pm 1$ and are finite.

    Moreover, $V$ satisfies: $V(x^*)=\frac{1}{1+\epsilon}$, $l>u>V$ on the set \bDGH{$(-1,1)$}, \bDGH{$V > u>l$} on the set \bDGH{$(-x^*,-1)\cup(1,x^*)$} and $u>V>l$ on the set $(-\infty,-x^*)\cup(x^*,\infty)$.

    As a result, $V$ is a strong solution of the HJB equation and satisfies (\ref{HJB1}) on $\mathbb{R}\setminus\{-1,1\}$.
\end{lemma}

\begin{proof}
    Given that the functions $l,u$ and $V$ are even functions, it suffices to prove the results on $[0,\infty)$.

    We require that $V$ is $C^1$ at $x = 1$ and $x=x^*$ and that $V(x^*)=\frac{1}{1 + \epsilon}$. These five conditions are sufficient to fix the five unknowns $A$, $B$, $C$, $D$ and $x^*$, and the values given in the lemma can be derived after some algebra.
\end{proof}

\begin{remark}
\label{rem:ItoTanaka}
    The candidate value function $V$ is not $C^2$ by Lemma \ref{lem:propertyV}. However, given that $V$ is $C^1$ and $V''$ fails to exist at a finite number of points, It\^{o}'s formula is still valid for $V(X_t)$, in the sense that $V(X_t)$ can still be written in the form $V(X_t) = V(x) + \int_0^t V'(X_s) dX_s + \int_0^t \frac{1}{2} V''(X_s) d[X]_s$, see \cite[Problem 3.6.24]{karatzas2014brownian}.
\end{remark}

\begin{lemma}
\label{lem:commonverification}
    The function $V$ is the value of the game under the common Poisson constraint, and $(\eta^C_{\bDGH{(-1,1)}}, \eta^C_{\bDGH{(}-x^*,-1)\cup(1,x^*\bDGH{)}})$ is a saddle point. That is, $v^C=V$, \bDGH{$A^C=(-1,1)$ and $B^C=(-x^*,-1)\cup(1,x^*)$}. \bDGH{Note that $A^C$ and $B^C$ are open.}
\end{lemma}

The proof of Lemma~\ref{lem:commonverification} is fairly standard (although it is for stopping under a Poisson constraint so not entirely so) and as our proof relies on a result from the next section it is given in the Appendix. The next result follows immediately from~Theorem~\ref{thm:necessity}.

\begin{corollary}
$V$ is also the value of the game under the independent Poisson constraint, and \bDGH{$(\eta^{M}_{(-1,1)}, \eta^{m}_{(-x^*,-1)\cup(1,x^*)})$} is a saddle point.
\end{corollary}

\begin{figure}[H]
\includegraphics[scale=0.5]{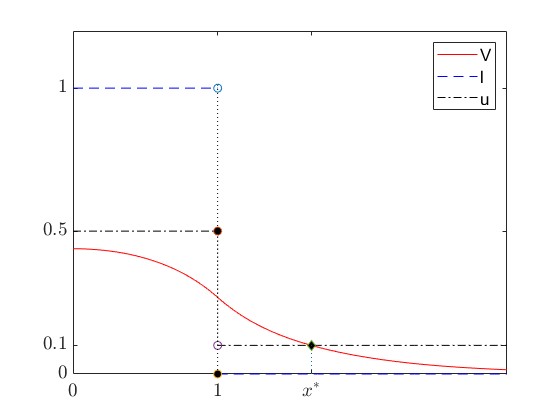}
\caption{The payoff functions and value of the game with payoff (\ref{eq:example}) under the common Poisson constraint, where we take $\lambda=r=1$ and $\epsilon=9$.}
\label{fig_sketch}
\end{figure}



\subsection{From the independent constraint to the common constraint: a counterexample}
Now we modify the game we solved in the previous subsection in such a way that the value functions for the independent and common constraint models no longer agree.

Let $V$ be as defined in \eqref{eq:value} with the constants as specified in Lemma~\ref{lem:propertyV}. Fix $\delta\in(0,\frac{x^*-1}{3})$ and define $\tilde{l}(x)=\1_{x\in \bDGH{(-1,1)}}+V(x^*-\delta)\1_{x\in(-x^*+2\delta,-1]\cup[1,x^*-2\delta)}$, and $\tilde{u}=u$.
Then, the payoff of the game is given by
\begin{eqnarray}\nonumber
    \tilde{R}(\tau,\sigma)&=&e^{-r\tau}(\1_{X_\tau\in \bDGH{(-1,1)}}+V(x^*-\delta)\1_{X_\tau\in(-x^*+2\delta,-1\bDGH{]}\cup\bDGH{[}1,x^*-2\delta)})\1_{\{\tau\leq \sigma\}}\\&&+e^{-r\sigma}\left(\frac{\phi^2-\theta^2}{\phi^2}\1_{X_\sigma\in[-1,1]}+\frac{1}{1+\epsilon}\1_{X_\sigma\notin [-1.1]}\right)\1_{\{\sigma<\tau\}}.\label{eq:payofftilde}
\end{eqnarray}
We want to argue that $V$ is the value of the game under the independent Poisson constraint but not under the common Poisson constraint.

Recall that $V$ is decreasing on $[0,\infty)$ and is an even function, therefore $V(x)>V(x^*-\delta)$ for $x\in (-x^*+2\delta,-1)\cup(1,x^*-2\delta)$. Hence $V>\tilde{l}>\tilde{u}$ holds on this set. Therefore the function $V$ satisfies the HJB equation $\mathcal{L}V-\frac{2\phi^2-\theta^2}{2}V +\frac{\phi^2-\theta^2}{2}\max(V,\tilde{l})+\frac{\phi^2-\theta^2}{2}\min(\tilde{u},V)=0$. Therefore, similar as the previous subsection, one can verify that $V$ is the value function of the game with payoff (\ref{eq:payofftilde}) under the independent Poisson constraint, and that $(\eta^{M}_{\bDGH{(-1,1)}}, \eta^{m}_{\bDGH{(-x^*,-1)\cup(1,x^*)}})$ is a saddle point. Hence, $v^I=V$, $A^I = \bDGH{(-1,1)}$ and $B^I= \bDGH{(-x^*,-1)\cup(1,x^*)}$.

\begin{figure}[H]
\includegraphics[scale=0.5]{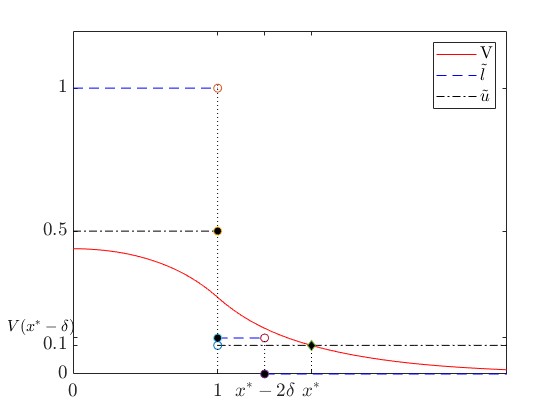}
\caption{The payoff functions and value of the game with payoff (\ref{eq:payofftilde}) under the independent Poisson constraint, where we take $\lambda=r=1$, $\epsilon=9$ and $\delta=\frac{x^*-1}{4}$. $V$ is not the value of the game under the common Poisson constraint as the set the set $\{x:\bDGH{V(x) \geq \tilde{l}(x)} >\tilde{u}(x)\}$ is nonempty. }
\label{fig_sketch2}
\end{figure}

However, $V$ does not satisfy $\mathcal{L}V-\frac{\phi^2}{2}V +\frac{\phi^2-\theta^2}{2}\max(\min(\tilde{u},V),\tilde{l})=0$. Indeed, on the set $(-x^*+2\delta,-1)\cup(1,x^*-2\delta)$, we have $\max(\min(\tilde{u},V),\tilde{l})=\tilde{l}>\frac{1}{1+\epsilon}$ by construction, and therefore $\mathcal{L}V-\frac{\phi^2}{2}V +\frac{\phi^2-\theta^2}{2}\max(\min(\tilde{u},V),\tilde{l})>0$ as $V$ satisfies $\mathcal{L}V-\frac{\phi^2}{2}V +\frac{\phi^2-\theta^2}{2} \frac{1}{1+\epsilon}=0$ by Lemma \ref{lem:propertyV}. Hence $V$ cannot be the value function of the game with payoff (\ref{eq:payofftilde}) under the common constraint and the \textsc{sup} player should deviate from $\eta^{M}_{\bDGH{(-1,1)}}$ and also stop on the set $(-x^*+2\delta,-1)\cup(1,x^*-2\delta)$. In particular, although $A^I$ and $B^I$ are disjoint, this does not imply $v^I=v^C$ or that $(\eta_{A^I}^C,\eta_{B^I}^C)$ is a saddle point.

\section{Existence of solutions for perpetual Dynkin games}
\label{sec:BSDE}

In this section, we establish existence and uniqueness results for the values of the Dynkin game under both common and independent constraints. In doing so, we verify that our standing assumptions are satisfied for a wide class of problems based on payoffs which are functions of an underlying diffusion. The mathematical innovation in this section is that we consider an infinite horizon BSDE: most of the extant literature considers a finite horizon BSDE which is typically simpler.

\subsection{Zero-sum Dynkin games under the common constraint}
\label{ssec:existenceC}


In this subsection we focus on the game under the common constraint. We will omit the superscript $C$ in this subsection for notational simplicity.

Instead of considering any filtration that supports the process $X$ and the Poisson process we work on a {minimal filtration to allow the application of martingale representation theorem}. Given the probability space $(\Omega,\mathcal{F},\mathbb{P})$ we define the filtration \(\mathbb{F}^* = \{\mathcal{F}^*_t\}_{t \geq 0}\) as the smallest filtration that contains both the natural filtration of $X$, \(\mathbb{F}^X=\{\mathcal{F}_t^X\}_{t\geq 0}\) and the natural filtration of the Poisson process, \(\mathbb{H}=\{\mathcal{H}_t\}_{t\geq 0}\), i.e., for each \(t \geq 0\),
$
\mathcal{F}^*_t = \sigma(\mathcal{F}_t^X \cup \mathcal{H}_t).
$
Finally, the filtration \(\mathbb{F}\) we use is the augmented version of \(\mathbb{F}^*\) chosen to satisfy the usual conditions of right-continuity and completeness with respect to \(\mathbb{P}\). 

For any $T\in (0,\infty]$, we define the spaces
\begin{eqnarray*}
    \mathcal{S}^2_{a,T}&=&\big\{\mathbb{F}\text{-progressively measurable processes}\ \mathbf{y}:  \lVert\mathbf{y}\rVert^2_{\mathcal{S}^2_{a,T}}<\infty\big\};\\
    \mathcal{H}^2_{a,T}&=&\big\{\mathbb{F}\text{-progressively measurable processes}\ Z:  \lVert Z \rVert^2_{\mathcal{H}^2_{a,T}}<\infty\big\}.
\end{eqnarray*}
where the weighted norms are defined as
\begin{equation*}
\lVert \mathbf{y}\rVert^2_{\mathcal{S}^2_{a,T}}=\E\left[\sup_{t\in[0,T]}e^{2at}\mathbf{y}_t^2\right];\quad
\lVert Z\rVert^2_{\mathcal{H}^2_{a,T}}=\E\left[\int_0^T e^{2at}Z_t^2 dt\right].
\end{equation*}
It is clear that for $T<\infty$, the above weighted norms are equivalent for any $a\in\mathbb{R}$. However, for $T=\infty$, the weighted norms become stronger as $a$ increases, and the corresponding space becomes smaller.

We work in the setting of a time-homogeneous diffusion process $X$ (with initial value $X_0=x$) and payoffs which are discounted functions of $X= X^x$.
\begin{assumption}\label{assumption:BSDE}
    For fixed initial value $x\in \sX$ the processes $l(X)=(l(X_t))_{t\geq 0}$ and $u(X)=(u(X_t))_{t\geq 0}$ satisfy the following:

    (1) $l(X),u(X)\in \mathcal{H}^2_{\alpha,\infty}$ for some $\alpha>-r$, where $r>0$ is the discount factor defined in Section 2.

    (2) {$\E[\sup_{t\geq 0}(L_t\vee U_t)]<\infty$}, and $\lim\limits_{t \uparrow \infty} L_t$ and $\lim\limits_{t \uparrow \infty} U_t$ both exist a.s.. We set $L_\infty = \lim\limits_{t \uparrow \infty} L_t$ and $U_\infty = \lim\limits_{t \uparrow \infty} U_t$.

\end{assumption}

Note that, (1) will be required for the BSDE argument. It will follow from Assumptions (1) and (2) that $L_\infty=0=U_\infty$, as we shall see in Lemma \ref{lem:UI}.

\begin{remark}
{A typical example in which the above assumption is satisfied is the following: let $l$ satisfy $|l(x)|\leq C(1+|x|)$, let $u$ satisfy $|u(x)|\leq C(1+|x|)$ and let $X$ be a geometric Brownian motion with $X_0=x$, drift $\mu$, and volatility $\sigma$ and suppose $r>\max\{\mu+\frac{1}{2}\sigma^2,0\}$.

Then we have $X_t=xe^{(\mu-\frac{1}{2}\sigma^2)t+\sigma W_t}$. It follows that $e^{2\alpha t}X_t^2=x^2e^{2(\alpha+\mu-\frac{1}{2}\sigma^2)t+2\sigma W_t}.$  Let $r>\max\{\mu+\frac{1}{2}\sigma^2,0\}$. In turn, {by an application of the inequality $|l(x)|^2\leq 2C^2(1+|x|^2)$,} $$\E[e^{2\alpha t}l(X_t)^2]\leq 2C^2\left(\exp(2\alpha t)+\exp\left(2\left(\alpha+\mu+\frac{1}{2}\sigma^2\right)t\right)\right).$$ It follows that $\E[\int_0^\infty e^{2\alpha t}l(X_t)^2dt]<\infty$ holds if $\mu+\frac{1}{2}\sigma^2+\alpha<0$ and $\alpha<0$, and further that $l(X)\in\mathcal{H}^2_{\alpha,\infty}$ for some $\alpha\in(-r,0)$.} 

{Furthermore, given that $\mu+\frac{1}{2}\sigma^2+\alpha<0$, it follows that $\mu<r$. The integrability of $\sup_tL_t$ is equivalent to the integrability of the maximum of a geometric Brownian motion with negative drift, which is immediate. It is also immediate that $L_t\rightarrow L_{\infty}=0$ a.s.. A similar argument applies to $u$ and $U$.
}\end{remark}

Consider the following infinite horizon BSDE defined on $[0,\infty)$,
\begin{equation}\label{eq:BSDE}
    d\mathbf{y}_t=-\left[\lambda \max\{l(X_t),\min(\mathbf{y}_t,u(X_t))\}-(\lambda+r)\mathbf{y}_t\right]\,dt+Z_t\,dW_t,\quad t\geq 0,
\end{equation}
subject to the asymptotic condition
\begin{equation}\label{boundary_cond}
\lim_{t\uparrow \infty}\E[e^{2\alpha t}\mathbf{y}_t^2]=0,
\end{equation}
where $\alpha$ is introduced in Assumption \ref{assumption:BSDE}. 

A solution to BSDE (\ref{eq:BSDE}) {subject to the asymptotic condition (\ref{boundary_cond})} is a pair of $\mathbb{F}$-progressively measurable processes $(\mathbf{y},Z)$ satisfying
\begin{equation*}
    \mathbf{y}_t=\mathbf{y}_T+\int_t^T\left[\lambda \max\{l(X_s),\min(\mathbf{y}_s,u(X_s))\}-(\lambda+r)\mathbf{y}_s\right]\,ds-\int_t^T Z_s\,dW_s
\end{equation*}
for $0\leq t\leq T<\infty$, and such that the asymptotic condition (\ref{boundary_cond}) holds. We aim to find a solution $(\mathbf{y},Z)$ in the spaces $ \mathcal{S}^2_{\alpha,\infty}\times\mathcal{H}^2_{\alpha,\infty}$.


The main idea behind solving {the BSDE (\ref{eq:BSDE}) subject to the asymptotic condition (\ref{boundary_cond})} is to first approximate it by a sequence of finite horizon BSDEs and establish uniform estimates for their solutions. The existence of a solution to the infinite-horizon problem (Proposition~\ref{prop:existence}) then follows from a fairly standard compactness argument. However, since we are dealing with unbounded solutions for infinite horizon BSDEs, whereas the majority of results in this direction only consider bounded solutions, we provide a proof in the Appendix. 
\begin{proposition}\label{prop:existence}
    Suppose that Assumption \ref{assumption:BSDE} holds. Then there exists a unique solution $(\mathbf{y},Z)\in\mathcal{S}^2_{\alpha,\infty}\times \mathcal{H}^2_{\alpha,\infty}$ to BSDE (\ref{eq:BSDE}) {subject to the asymptotic condition (\ref{boundary_cond})}.
\end{proposition}

\begin{lemma}\label{lem:DPE0}
   Let $(\mathbf{y},Z)$ be the unique solution of (\ref{eq:BSDE}) {subject to the asymptotic condition (\ref{boundary_cond})}. Then for $n\geq 0$, $\mathbf{y}$ is also a solution to the following recursive equation for $n\geq 0$:
\begin{equation}\label{eq:DPE0}
    e^{-rT_n}\mathbf{y}_{T_n}=\E[e^{-rT_{n+1}}\max\{l(X_{T_{n+1}}),\min(\mathbf{y}_{T_{n+1}},u(X_{T_{n+1}}))\}|\mathcal{F}_{T_n}].
\end{equation}
\end{lemma}
\begin{proof}
    {Given that $(\mathbf{y},Z)\in\mathcal{S}^2_{\alpha,\infty}\times \mathcal{H}^2_{\alpha,\infty}$ with $\alpha>-r$ and that the weighted norms $\lVert \cdot \rVert^2_{\mathcal{S}^2_{a,\infty}}$ and $\lVert \cdot \rVert^2_{\mathcal{H}^2_{a,\infty}}$ become stronger as $a$ increases, it follows that $(\mathbf{y},Z)\in\mathcal{S}^2_{-r,\infty}\times \mathcal{H}^2_{-(\lambda+r),\infty}$.}

    Applying It\^{o}'s formula, we obtain the following expression for $T>t\geq 0$:
\[e^{-(\lambda+r)t}\mathbf{y}_t=e^{-(\lambda+r)T}\mathbf{y}_T+\int_t^Te^{-(\lambda+r)s}\lambda\max\{l(X_s),\min(\mathbf{y}_s,u(X_s))\}\,ds-\int_t^Te^{-(\lambda+r)s}Z_s\,dW_s.\]

By the result that $Z\in\mathcal{H}^2_{-(\lambda+r),\infty}$, the stochastic integral term is square integrable hence is a uniformly integrable martingale. For $T>T_n$, utilizing the density of $T_{n+1}-T_n$ conditional on $\mathcal{F}_{T_n}$, we have:
\begin{align*}
    e^{-rT_n}\mathbf{y}_{T_n} &= \E\bigg[e^{-\lambda(T - T_n)} e^{-rT} \mathbf{y}_T \\
    &\quad + \int_{T_n}^T e^{-\lambda(s - T_n)} e^{-rs} \lambda 
    \max\{ l(X_s), \min(\mathbf{y}_s, u(X_s)) \} \, ds \, \bigg| \, \mathcal{F}_{T_n}\bigg] \\
    &= \E\bigg[ e^{-rT} \mathbf{y}_T \, \mathbf{1}_{\{T_{n+1} \geq T\}} \\
    &\quad + e^{-rT_{n+1}} \lambda 
    \max\{ l(X_{T_{n+1}}), \min(\mathbf{y}_{T_{n+1}}, u(X_{T_{n+1}})) \} 
    \mathbf{1}_{\{T_{n+1} < T\}} \, \bigg| \, \mathcal{F}_{T_n} \bigg].
\end{align*}
{Since $\mathbf{y}\in \mathcal{S}^2_{-r,\infty}$, the term $\E[e^{-rT}\mathbf{y}_T\1_{\{T_{n+1}\geq T\}}]$ vanishes as we take $T\uparrow \infty$. Thus, the result follows by {monotone convergence.}}
\end{proof}

We have proved that, under Assumption \ref{assumption:BSDE}, there exists a solution $(\mathbf{y}_{T_n})_{n\geq 0}$ to the recursive equation (\ref{eq:DPE0}). In the remaining part of this section, we will show that the solution to the recursive equation (\ref{eq:DPE0}) defines the value of the game with payoff $R$ under the common constraint, and an assumption that $M_\infty=0$. (Note that by Lemma~\ref{lem:UI} $L_\infty=0=U_\infty$ so that in this case the only natural candidate value for $M_\infty$ is zero.) 


{We need uniform integrability for the verification of game value, and this is covered by the following lemma:}

\begin{lemma}\label{lem:UI}
   Let $(\mathbf{y},Z)$ be the unique solution of BSDE (\ref{eq:BSDE}) {subject to the asymptotic condition (\ref{boundary_cond})}. {Define $Y_{t}=e^{-rt}\mathbf{y}_t$ and $\hat{Y}_{t}=\max\{\min\{U_{t}, Y_{t}\},L_{t}\}$.

   Then the following results hold:

   (1) $\E\left[\sup_{t\in[0,\infty)}Y_t\right]<\infty$; $\E\left[\sup_{t\in[0,\infty)}\hat{Y}_t\right]<\infty$.

   (2) $L_t,U_t,Y_t,\hat{Y}_t\rightarrow 0$ a.s. as $t\rightarrow \infty$.}

\end{lemma}
\begin{proof}
    See Appendix.
\end{proof}






In the game under the common constraint, the game starts at time $0$ and players are not allowed to stop immediately, so $T_1$ is the first time when the players can stop. We want to consider an auxiliary game with the same payoff as defined in (\ref{eq:gamepayoff}), but allowing agents an additional opportunity to stop at time 0.


Let $T_0=0$ and let $\mathcal{R}_0(\lambda) = \mathcal{R}(\lambda) \cup \{ T_0 \}$. Then the upper and lower values of this auxiliary game are:
\begin{eqnarray}
    \overline{\nu}_{0}&=&\inf_{\sigma\in\mathcal{R}_{0}(\lambda)}\sup_{\tau\in\mathcal{R}_{0}(\lambda)}\E[R(\tau,\sigma) ];\\
    \underline{\nu}_{0}&=&\sup_{\tau\in\mathcal{R}_{0}(\lambda)}\inf_{\sigma\in\mathcal{R}_{0}(\lambda)}\E[R(\tau,\sigma) ].
\end{eqnarray}

We will consider the dynamic versions of the original and auxiliary games. Consider the problem where two players aim to maximize/minimize $R(\tau,\sigma)$ conditioning on the information at time $T_k$. Given any $\tau,\sigma$, the outcome of the problem is a random variable, and should be maximized/minimized in the sense of essential supremum/infimum.
For $k \geq 1$ let $\mathcal{R}_k(\lambda)=\{\gamma: \gamma\, \textrm{ is  an}\, \mathbbm{F}\textrm{-stopping time such that } \gamma(\omega)=T_n(\omega)\, \textrm{for some}\, n\in \{k, \dots \infty \}\}$. (Note that putting $k=0$ in this definition recovers $\mathcal{R}_0$ defined above.)
Then the set of admissible strategies for the dynamic version of the original game is $\mathcal{R}_{k+1}(\lambda)$, and the set of admissible strategies for the dynamic version of the auxiliary game is $\mathcal{R}_{k}(\lambda)$. Hence we define the following upper and lower values for the dynamic version of the original game as:
\begin{eqnarray}
    \overline{{v}}_{T_k}&=&\essinf_{\sigma\in \mathcal{R}_{k+1}(\lambda)}\esssup_{\tau\in \mathcal{R}_{k+1}(\lambda)}\E[R(\tau,\sigma)|\mathcal{F}_{T_{k}}] \label{eq:dynamicpayoffo1};\\
    \underline{{v}}_{T_k}&=&\esssup_{\tau\in \mathcal{R}_{k+1}(\lambda)}\essinf_{\sigma\in \mathcal{R}_{k+1}(\lambda)}\E[R(\tau,\sigma)|\mathcal{F}_{T_{k}}] \label{eq:dynamicpayoffo2}.
\end{eqnarray}
We also define the upper and lower values for the dynamic version of the auxiliary game as:
\begin{eqnarray}
    \overline{\nu}_{T_k}&=&\essinf_{\sigma\in \mathcal{R}_k(\lambda)}\esssup_{\tau\in \mathcal{R}_k(\lambda)}\E[R(\tau,\sigma)|\mathcal{F}_{T_{k}}] \label{eq:dynamicpayoffa1};\\
    \underline{\nu}_{T_k}&=&\esssup_{\tau\in \mathcal{R}_k(\lambda)}\essinf_{\sigma\in \mathcal{R}_k(\lambda)}\E[R(\tau,\sigma)|\mathcal{F}_{T_{k}}] \label{eq:dynamicpayoffa2}.
\end{eqnarray}
\begin{lemma}\label{lem:auxiliaryvalue}
    Let $(\mathbf{y},Z)$ be the unique solution of (\ref{eq:BSDE}). Let $Y_{t}$ and $\hat{Y}_{t}$ be as defined in Lemma \ref{lem:UI}, with $\hat{Y}_\infty=0$.

    Then $\hat{Y}_{T_k}=\overline{\nu}_{T_k}=\underline{\nu}_{T_k}$, where $\overline{\nu}_{T_k}$ and $\underline{\nu}_{T_k}$ are as defined in (\ref{eq:dynamicpayoffa1}) and (\ref{eq:dynamicpayoffa2}), respectively. Moreover, we have $\hat{Y}_{T_k}=\E[R(\tau^*_k,\sigma^*_k)|\mathcal{F}_{T_k}]$, where ${\sigma}_{k}^*=\inf\{T_N\geq T_{k}:Y_{T_N}>  U_{T_N}\}$ and ${\tau}_{k}^*=\inf\{T_N\geq T_{k}:Y_{T_N}< L_{T_N}\,\textrm{or}\,\,Y_{T_N}\geq L_{T_N}>U_{T_N}\}$.
\end{lemma}

\begin{proof}
    We will later establish that:

    (1)$(\hat{Y}_{T_n\wedge\tau\wedge {\sigma}^*_k})_{n\geq k}$ is a {uniformly integrable} supermartingale for any $\tau\in \mathcal{R}_k(\lambda)$.

    (2)$(\hat{Y}_{T_n\wedge{\tau}_k^* \wedge \sigma})_{n\geq k}$ is a uniformly integrable submartingale for any $\sigma \in \mathcal{R}_k(\lambda)$.
Before proving these properties, we will first show that conditions (1) and (2) are sufficient for the desired result. Note that, by Lemma \ref{lem:UI}, $R(\tau,\sigma)=\hat{Y}_{\sigma\wedge \tau}=0$ on the set $\{\tau=\sigma=\infty\}$. 

Observe that, on the set $\{\tau\leq \sigma_k^*\}$, $\hat{Y}_\tau\geq L_\tau$ holds by the definition of $\hat{Y}_\tau=\max\{\min\{U_{\tau}, Y_{\tau}\},L_{\tau}\}$. Also, we have \bDGH{$Y_{\sigma^*_k}> U_{\sigma^*_k}$} on the set $\{\sigma_k^*<\tau\}$ by the definition of $\sigma^*_k$, thus $\hat{Y}_{\sigma^*_k}=\max\{\min\{U_{\sigma^*_k}, Y_{\sigma^*_k}\},L_{\sigma^*_k}\}\geq U_{\sigma^*_k}$ holds on this set. Hence, using the property that $(\hat{Y}_{T_n\wedge\tau\wedge {\sigma}^*_k})_{n\geq k}$ is a {uniformly integrable} supermartingale, the optional stopping theorem implies:  
\begin{eqnarray}\label{eq:taustar}
    \hat{Y}_{T_{k}}&\geq& \E[\hat{Y}_{\tau\wedge{\sigma}^*_{k}}|\mathcal{F}_{T_{k}}]\nonumber\\
    &\geq& \E[L_\tau\1_{\tau\leq \sigma^*_{k}}+U_{\sigma^*_{k}}\1_{\sigma^*_{k}<\tau}|\mathcal{F}_{T_{k}}]\nonumber\\
    &=&\E[R(\tau,\sigma^*_{k})|\mathcal{F}_{T_{k}}].
\end{eqnarray}

Next, we have both $Y_\sigma\geq L_\sigma$ and $U_\sigma\geq L_\sigma$ on the set $\{\sigma<\tau^*_k\}$, hence $\hat{Y}_\sigma\leq U_\sigma$ holds by the identity $\hat{Y}_{\sigma}=\max\{\min\{U_{\sigma}, Y_{\sigma}\},L_{\sigma}\}$. We also have $\hat{Y}_{\tau^*_k}\leq L_{\tau^*_k}$ on the set $\{\tau^*_k\leq \sigma\}$. Indeed, by the definition of $\tau^*_k$, there are three possible cases about the order among $Y_{\tau^*_k}, L_{\tau^*_k}$ and $ U_{\tau^*_k}$: $Y_{\tau^*_k}< L_{\tau^*_k}\leq U_{\tau^*_k}$, $Y_{\tau^*_k}< L_{\tau^*_k} \textrm{ with } U_{\tau^*_k}<L_{\tau^*_k}$, and $Y_{\tau^*_k}\geq L_{\tau^*_k}>U_{\tau^*_k}$
. In any of these three cases, $\min\{U_{\tau^*_k}, Y_{\tau^*_k}\}\leq L_{\tau^*_k}$ holds, hence $\hat{Y}_{\tau^*_k}\leq L_{\tau^*_k}$ holds by the identity $\hat{Y}_{\tau^*_k}=\max\{\min\{U_{\tau^*_k}, Y_{\tau^*_k}\},L_{\tau^*_k}\}$.
Hence, using the property that $(\hat{Y}_{T_n\wedge{\tau}_k^* \wedge \sigma})_{n\geq k}$ is a uniformly integrable submartingale, the optional stopping theorem implies:
\begin{eqnarray}\label{eq:sigmastar}
    \hat{Y}_{T_{k}}&\leq& \E[\hat{Y}_{\tau^*_k\wedge{\sigma}}|\mathcal{F}_{T_{k}}]\nonumber\\
    &\leq& \E[L_{\tau^*_k}\1_{\tau^*_k\leq \sigma}+U_{\sigma}\1_{\sigma<\tau^*_{k}}|\mathcal{F}_{T_{k}}]\nonumber\\
    &=&\E[R(\tau^*_{k},\sigma)|\mathcal{F}_{T_{k}}].
\end{eqnarray}
The inequality (\ref{eq:taustar}) implies that $\hat{Y}_{T_k}\geq \esssup_{\tau\in \mathcal{R}_k(\lambda)}\E[R(\tau,\sigma^*_k)|\mathcal{F}_{T_{k}}]\geq \overline{\nu}_{T_k}$. Similarly, we have $\hat{Y}_{T_k}\leq \underline{\nu}_{T_k}$ by (\ref{eq:sigmastar}). It is obvious that $\overline{\nu}_{T_k}\geq \underline{\nu}_{T_k}$, therefore the equality $\hat{Y}_{T_k}=\overline{\nu}_{T_k}=\underline{\nu}_{T_k}$ holds and $\hat{Y}_{T_k}=\E[R(\tau^*_k,\sigma^*_k)|\mathcal{F}_{T_k}]$ is immediate.


It remains to prove properties (1) and (2). Uniform integrability of $\hat{Y}$ follows by Lemma \ref{lem:UI}. To prove the supermartingale property, we fix $n\geq k$. We have:
\[\E[\hat{Y}_{T_{n+1}\wedge\tau\wedge \sigma^*_k}|\mathcal{F}_{T_n}]=\1_{\tau\wedge \sigma^*_k\leq T_n}\hat{Y}_{T_n\wedge \tau \wedge \sigma^*_{k}}+\1_{\tau\wedge \sigma^*_k\geq T_{n+1}}\E[\hat{Y}_{T_{n+1}}|\mathcal{F}_{T_{n}}].\]
Observe that, on the set $\{\tau\wedge \sigma^*_k\geq T_{n+1}\}$, we have $T_n<\sigma^*_k$. Therefore \bDGH{$U_{T_n}{\geq } {Y}_{T_n}$} holds on this set, hence $\hat{Y}_{T_n}=\max\{Y_{T_n},L_{T_n}\}\geq Y_{T_n}=\E[\hat{Y}_{T_{n+1}}|\mathcal{F}_{T_n}]$ holds on this set, and this implies the supermartingale property.

For the submartingale property, we have $T_n<\tau^*_k$ on the set $\{\tau^*_k\wedge \sigma\geq T_{n+1}\}$, which implies that $\min\{U_{T_n},Y_{T_n}\}\geq L_{T_n}$. It follows that $\hat{Y}_{T_n}\leq Y_{T_n}$ holds on this set, hence the submartingale property follows.
\end{proof}

\begin{theorem}\label{thm:common}
    Let $(\mathbf{y},Z)$ be the unique solution of BSDE (\ref{eq:BSDE}), and let $Y_{t}$ be as defined in Lemma \ref{lem:UI}.

    Then $Y_{T_k}=\overline{{v}}_{T_k}=\underline{{v}}_{T_k}$, where $\overline{{v}}_{T_k}$ and $\underline{{v}}_{T_k}$ are as defined in (\ref{eq:dynamicpayoffo1}) and (\ref{eq:dynamicpayoffo2}), respectively. Moreover, we have $Y_{T_k}=\E[R(\tau_{k+1}^*,\sigma_{k+1}^*)|\mathcal{F}_{T_k}]$, where $\tau_{k+1}^*,\sigma_{k+1}^*$ are as defined in Lemma \ref{lem:auxiliaryvalue}.

    In particular, $Y_0=\mathbf{y}_0$ is the value of the Dynkin game with payoff $R$ under the common Poisson constraint, and $(\tau^*_1,\sigma^*_1)$ is a saddle point of this game. 
\end{theorem}

\begin{proof}
    Observe that, for any $\tau \in \mathcal{R}_{k+1}(\lambda)$,
\begin{equation*}
    Y_{T_k}=\E[\hat{Y}_{T_{k+1}}|\mathcal{F}_{T_k}]\geq\E[\E[R(\tau,\sigma^*_{k+1})|\mathcal{F}_{T_{k+1}}]|\mathcal{F}_{T_{k}}]]
    =\E[R(\tau,\sigma^*_{k+1})|\mathcal{F}_{T_k}],
\end{equation*}
where the first equality is by (\ref{eq:DPE0}) and the inequality is by (\ref{eq:taustar}). Similarly, {by (\ref{eq:sigmastar}),} $Y_{T_k}\leq \E[R({\tau}^*_{k+1},\sigma)|\mathcal{F}_{T_k}]$ holds for any $\sigma \in \mathcal{R}_{k+1}(\lambda)$. Hence, $Y_{T_k}=\overline{{v}}_{T_k}=\underline{{v}}_{T_k}=\E[R(\tau_{k+1}^*,\sigma_{k+1}^*)|\mathcal{F}_{T_k}]$.

By taking $k=0$, we get $Y_{T_0}=\mathbf{y}_0$ and $J^x(\tau,\sigma^*_1)\leq \mathbf{y}_0\leq J^x(\tau^*_1,\sigma)$ for any $\tau,\sigma \in\mathcal{R}_1(\lambda)$. Hence $\mathbf{y}_0$ is the value and $(\tau^*_1,\sigma^*_1)$ is a saddle point of the game under the common constraint.
\end{proof}

\begin{theorem}
\label{thm:satisfyassC}
Suppose that $X$, $l$ and $u$ are such that Assumption~\ref{assumption:BSDE} holds for each $x \in \sX$ and suppose that $M_\infty=0$.
Then Assumptions~\ref{assumption:existence}, \ref{assumption:technical} \bDGH{and \ref{assumption:gamevalue}(i)} are satisfied.
\end{theorem}

\begin{proof}
Given that the driver does not directly depend on $t$ the solution to the infinite horizon BSDE (\ref{eq:BSDE}) has a Markovian representation, which does not depend on $t$. This defines a function $v^C$ on $\sX$ such that $\mathbf{y}_t^x=v^C(X_t^x)$, where $\mathbf{y}^x$ denotes the solution of the system (\ref{eq:BSDE}) for initial value $X_0=x$.

By Theorem \ref{thm:common} and the Markovian representation $v^C$, conditioning on any initial value $x$, $v^C(x)$ is the value of the Dynkin game with payoff $R$ under the common constraint, and the saddle point $(\tau_1^*,\sigma_1^*)$ defined in Theorem \ref{thm:common} has the form ${\sigma}_{1}^*=\inf\{T_N\geq T_{1}:v^C(X_{T_N})>  u(X_{T_N})\}$ and ${\tau}_{1}^*=\inf\{T_N> T_{1}:v^C(X_{T_N})< l(X_{T_N})\,\textrm{or}\,v^C(X_{T_N})\geq l(X_{T_N})>u(X_{T_N})\}$. Hence Assumption \ref{assumption:existence} is satisfied.

Since $M_\infty=0$, Assumption~\ref{assumption:technical} is trivially satisfied. \bDGH{Assumption \ref{assumption:gamevalue}(i) follows by Lemma \ref{lem:DPE0}}.
\end{proof}

It follows from the results of this section that there are a wide class of examples such that Assumptions
\ref{assumption:existence} and \ref{assumption:technical} hold. It should be noted however, that the class of examples for which  Assumptions
\ref{assumption:existence} and \ref{assumption:technical} hold is much wider than the class covered by Assumption~\ref{assumption:BSDE}. Example~\ref{eg:extension} provides an example in this wider class.

\subsection{Zero-sum Dynkin games under the independent constraint}
\label{ssec:existenceI}

The results for the game for the independent constraint are very similar. We summarize the most relevant conclusions here.

\begin{theorem}
\label{thm:satisfyassi}
Suppose that $X$, $l$ and $u$ are such that Assumption~\ref{assumption:BSDE} holds for each $x \in \sX$ and suppose that $M_\infty=0$.
Then Assumptions~\ref{assumption:existence1}, \ref{assumption:technical1} and \ref{assumption:gamevalue}(ii) are satisfied.
\end{theorem}

\begin{proof}
The proof follows the proof of all the lemmas and propositions of Section~\ref{ssec:existenceC}, leading ultimately to the analogous conclusion to Theorem~\ref{thm:satisfyassi}, namely this theorem. There are two main changes: firstly, rather than considering the BSDE \eqref{eq:BSDE} we consider the infinite horizon BSDE
\[ 
    d\mathbf{y}_t=-\left[\lambda \max\{l(X_t),\mathbf{y}_t\} + \lambda \min \{ \mathbf{y}_t,u(X_t)\} -(2\lambda+r)\mathbf{y}_t\right]\,dt+Z_t\,dW_t,\quad t\geq 0,
\] 
The associated driver is $\tilde{f}(t,y) = \lambda \max(l(X_t), y) + \min(y,u(X_t)) - (2 \lambda + r)y$. Since this is still Lipschitz the proofs of the corresponding results pass through unchanged and the BSDE has a unique solution in the spaces $ \mathcal{S}^2_{\alpha,\infty}\times\mathcal{H}^2_{\alpha,\infty}$. 

Secondly, in the verification of the solution to the auxiliary games (which corresponds to Lemma \ref{lem:auxiliaryvalue} under the common constraint), we need to merge the two Poisson sequences and this defines an increasing sequence $(T_k^{mer})_{k\geq 1}$ (See \cite[Section 3]{liang2020risk} for details).  As a result we define $\hat{Y}_{T^{mer}_k}=\max\{L_{T^{mer}_k},Y_{T^{mer}_k}\}\1_{T^{mer}_k\in T^{M}}+\min\{U_{T^{mer}_k},Y_{T^{mer}_k}\}\1_{T^{mer}_k\in T^{m}}$, where $T^{mer}_k\in T^{M}$ denotes the event that the signal $T^{mer}_k$ is the \textsc{sup} player's opportunity. {\em Mutatis mutandis} the remainder of the argument is the same.
\end{proof}

\section*{Acknowledgment}
We are grateful to the editor and the three referees whose
constructive comments led to a substantial
enhancement of the paper.



\appendix

\section{Alternative conventions on the choice of $M$}
\label{app:m=u}

\subsection{The case $M=U$}

In the current and the next subsection we consider an extension of our main example. Recall that \bDGH{in the main part of the paper} we assumed that $L_t = e^{-r t} l(X_t)$ and $U_t = e^{-r t} u(X_t)$, and the payoff when the two players stop simultaneously is $M=L$. 
Throughout this section we assume that $M=U$. We further suppose that $l$ is lower semicontinuous and $u$ is upper semicontinuous.
\begin{theorem}
Suppose that the Dynkin game under the common Poisson constraint satisfies Assumption \ref{assumption:technical} and has a value function $\{ v^C(x) \}_{x \in \sX}$, with saddle point $(\eta^C_{A^C},\eta^C_{B^C})$, where 
     $A^C=\{x:v^C(x)< l(x) \}$, $B^C=\{x:v^C(x)> u(x)\}\cup\{x:l(x)> u(x)\geq v^C(x)\}$.
Suppose that $v^C$ is continuous.

Then the following are equivalent.
\begin{enumerate}
    \item $v^C \vee u \geq l$.
    \item $A^C \cap B^C = \emptyset$.
    \item $J^x(\eta^{M}_{A^C},\eta_{B^C}^{m})=J^x(\eta^C_{A^C},\eta^C_{B^C})$ for every $x\in \sX$.   
\end{enumerate}
Moreover, if any of the above conditions hold, then the Dynkin game under the independent Poisson constraint has the same solution as the game under the common Poisson constraint. That is, $v^I(x)=v^C(x)$ for every $x\in \sX$, and $(\eta^{M}_{A^C},\eta^{m}_{B^C})$ is a saddle point of the game under the independent Poisson constraint. 
\end{theorem}





\begin{theorem}
Suppose that the Dynkin game under the independent Poisson constraint satisfies Assumption \ref{assumption:existence1} and Assumption \ref{assumption:technical1}  and let the solution be given by $(v^I, A^I, B^I)$ where 
$A^I=\{x:v^I(x)< l(x) \}$, $B^I=\{x:v^I(x)> u(x)\}$.
Suppose that $v^I$ is continuous.

Then the following are equivalent.
\begin{enumerate}
    \item  $v^I \vee u \geq l$.
    \item $A^I \cap B^I = \emptyset$ and $\{x: l(x)>u(x)\geq v^I(x)\}=\emptyset$.
    \item $J^x(\eta^{C}_{A^I},\eta_{B^I}^{C})=J^x(\eta^M_{A^I},\eta^m_{B^I})$ for every $x\in \sX$.
\end{enumerate}
Moreover, if any of the above conditions hold, then the Dynkin game under the common Poisson constraint has the same solution as the game under the independent Poisson constraint. That is, $v^C(x)=v^I(x)$ for every $x\in \sX$, and $(\eta^{C}_{A^I},\eta^{C}_{B^I})$ is a saddle point of the game under the common Poisson constraint.    
\end{theorem}

\begin{proposition}
Suppose that the value $v^C(x)$ of the game under the common Poisson constraint exists for every $x\in \sX$ and satisfies:
    \[v^C(x)=\E^x[e^{-rT_1^C}\min\{u(X_{T_{1}^C}),\max(v^C(X_{T_{1}^C}),l(X_{T_{1}^C}))\}].\]
Suppose that the value $v^I(x)$ of the game under the independent Poisson constraint satisfies Assumption \ref{assumption:gamevalue} (ii).

Suppose that $v^I=v^C$ and $v$ is continuous where $v=v^I \equiv v^C$. Then $v\vee u \geq l$.

\end{proposition}

\subsection{Other choices for the value of $M$}

In this subsection we consider \bDGH{a more general class of} functions $m$ \bDGH{for which the analysis behind Theorem~\ref{thm:necessity} can still be applied.}

Suppose that $l$ is lower semi-continuous and $u$ is upper semi-continuous and that $l$ and $u$ are given and fixed. Then $E:= \{ x: l(x)>u(x) \}$ can be written as a countable union of open intervals $E = \cup_{k \geq 1} E_k$.
\begin{definition}
Let ${\mathcal M}^0 = \{ \mbox{$m$: for each $k \geq 1$ either $m=l$ on $E_k$ or $m=u$ on $E_k$} \}$, and let ${\mathcal M} = \{ m \in {\mathcal M}^0: l \wedge u \leq m \leq l \vee u \}$. 
\end{definition}

If $m=l$ everywhere, or $m=u$ everywhere then clearly $m \in {\mathcal M}$. Note also that on the set $\{x : l(x)\leq u(x) \}$ the value of $m$ 
is only constrained by the order condition $l \leq m \leq u$.

The proof of the following theorem is similar to the proof of Theorem~\ref{thm:necessity}.

\begin{theorem}\label{thm:mixedmain}     Suppose that $l$ is lower semicontinuous and $u$ is upper semicontinuous and \bDGH{that $m \in {\mathcal M}$}.

Suppose that the Dynkin game under the common Poisson constraint satisfies Assumption \ref{assumption:technical} and has a value function $\{ v^C(x) \}_{x \in \sX}$, with saddle point $(\eta^C_{A^C},\eta^C_{B^C})$, where 
     $A^C=\{x:v^C(x)< l(x) \}\cup\{x:v^C(x) \geq l(x)> u(x),m(x)=l(x)\}$, $B^C=\{x:v^C(x)> u(x)\}\cup\{x:l(x)> u(x)\geq v^C(x),m(x)=u(x)\}$.
Suppose that $v^C$ is continuous.

Then the following are equivalent.
\begin{enumerate}
    \item $v^C \vee u \geq l$ on the set $\{x:m(x)=u(x)<l(x)\}$, and $v^C\wedge l\leq u$ on the set $\{x:m(x)=l(x)>u(x)\}$.
    \item $A^C \cap B^C = \emptyset$. 
    \item $J^x(\eta^{M}_{A^C},\eta_{B^C}^{m})=J^x(\eta^C_{A^C},\eta^C_{B^C})$ for every $x\in \sX$.
\end{enumerate}
Moreover, if any of the above conditions hold, then the Dynkin game under the independent Poisson constraint has the same solution as the game under the common Poisson constraint. That is, $v^I(x)=v^C(x)$ for every $x\in \sX$, and $(\eta^{M}_{A^C},\eta^{m}_{B^C})$ is a saddle point of the game under the independent Poisson constraint.
\end{theorem}










    





\section{Proofs}

[Proof of Lemma~\ref{lem:comp1}]

We prove that $(\tau\wedge \eta_E^{M},X_{\tau\wedge \eta_E^{M}})$ and $(\tau\wedge\eta_E^{m},X_{\tau\wedge\eta_E^{m}})$ have the same distribution using coupling of marked Poisson processes. The proof of the second result follows in an identical fashion.

We expand the probability space $(\Omega, \mathcal{F}, \mathbb{F}, \mathbb{P})$ if necessary so that it supports a marked point process $N$ taking values in $\mathbb{R}_+ \times Y$ with rate $dt \times  \mu(dy)$, and such that $X$ and $N$ are independent. Here $\mu$ is a measure on $Y$.

If $Y$ is a singleton $Y= \{ y \}$ then $\mu$ is a point mass $\mu = \mu_Y \delta_y$ where $\mu_Y=\mu(\{y\}) \in (0,\infty)$ and we may identify events of the marked point process with events of a standard Poisson process with rate $\mu_Y$.

We are interested in the case where $Y$ is a two-point set, $Y= \{R,B\}$ where we will refer to a mark with label $R$ (respectively $B$) as a red (respectively blue) mark. Suppose $\mu$ is such that $\mu(\{R\})= \mu_R$ and $\mu(\{B\}) = \mu_B$. By the thinking of the previous paragraph, we can identify the event times of the marked point process where we consider all the marks as a standard Poisson process with rate $\mu_R + \mu_B$ and event times $(T^{\mu_R + \mu_B}_k)_{k \geq 1}$. If we only consider the Poisson process generated by the red (respectively blue) marks then we have a Poisson process of rate $\mu_R$ ($\mu_B$) with event times $(T^{\mu_R}_k)_{k \geq 1}$ ($(T^{\mu_B}_k)_{k \geq 1}$).

Using this coupling it is clear that for each $\omega\in \Omega$ we have $\{ T^{\mu_R}_k(\omega) \}_{k \geq 1} \subseteq \{ T^{\mu_R+ \mu_B}_k(\omega) \}_{k \geq 1}$ and
$\{ T^{\mu_R}_k(\omega) \}_{k \geq 1} \cup \{ T^{\mu_B}_k(\omega) \}_{k \geq 1} = \{ T^{\mu_R+ \mu_B}_k(\omega) \}_{k \geq 1}$. Moreover, the set on the left of the last equation is a disjoint union (almost surely).

Take $\mu_R=\mu_B=\lambda$. We use the red marks to denote the Poisson signals of the \textsc{sup} player under the independent constraint, so that their stopping time $\tau$ can be defined using red marks. Note that, given that $\mathbb{P}^x(\mbox{$\tau=\infty$ or $X_\tau\in D$})=1$, only those red marks at which $X$ is in $D$ are used for the modelling of $\tau$.

Consider the following two approaches:

    (1) Thin out all the blue marks from the marked Poisson process, then find the first event time of the remaining marked process such that $X$ is in the set $E$; call this time $\eta^{(1)}$. 

    (2) Thin out all the blue marks at which the process $X$ is in the set $D$, and all the red marks at which $X$ is not in $D$, then find the first event time of the remaining Poisson process such that $X$ is in the set $E$; call this time $\eta^{(2)}$. 

    In either approach above we never thinned any red marks at which $X$ is in $D$, and so the definition of $\tau$  is not affected. It is immediate that the event time $\tau \wedge \eta^{(1)}$ defined via the first approach has the same distribution as $\tau\wedge \eta_E^{M}$, and the event time $\tau \wedge \eta^{(2)}$ defined via the second approach has the same distribution as $\tau\wedge \eta_E^{m}$.  But the two approaches are equivalent in the sense that in each approach half of the marks are thinned from a Poisson process with intensity $2\lambda$, and the first arrival time of $E$ is defined using the remaining process up to time $\tau$.

    Hence $(\tau\wedge \eta_E^{M},X_{\tau\wedge \eta_E^{M}})$ and $(\tau\wedge\eta_E^{m},X_{\tau\wedge\eta_E^{m}})$ have the same distribution. 

For the expected payoff, observe that we can write $R(\tau, \eta_E^{m})$ as
\begin{eqnarray*}
R(\tau, \eta_E^{m}) & = & e^{-r(\tau\wedge \eta_E^{m})}(l(X_{\tau\wedge \eta_E^{m}})\1_{X_{\tau\wedge \eta_E^{m}}\in D}  \\
&& \hspace{5mm} +u(X_{\tau\wedge \eta_E^{m}})\1_{X_{\tau \wedge \eta_E^{m}}\in E})\1_{\tau\wedge \eta_E^{m}<\infty}+M_\infty\1_{\tau\wedge \eta_E^{m}=\infty},
\end{eqnarray*}
and we can write $R(\tau, \eta_E^{M})$ in the exact same form on replacing $\tau\wedge \eta_E^{m}$ by $\tau\wedge \eta_E^{M}$. Therefore $J^x(\tau,\eta_E^{m})=J^x(\tau,\eta_E^{M})$ holds given that $(\tau\wedge \eta_E^{M},X_{\tau\wedge \eta_E^{M}})$ and $(\tau\wedge\eta_E^{m},X_{\tau\wedge\eta_E^{m}})$ have the same distribution, and that $M_\infty$ is independent of the Poisson signals. 
$\square$

[Proof of Lemma \ref{lem:SMPtau}]

We prove the first equality. The second equality can be proved via the same approach. Recall that $A=A^C$ and $B=B^C$. Let $\theta_t$ be the shift operator on the canonical space and let $\theta_\tau(\omega)=\theta_{\tau(\omega)}(\omega)$.


By Corollary \ref{cor:commonequality} we have $v^C(X^x_\tau)=J^{X^x_\tau}(\eta_A^C,\eta_B^C)=J^{X^x_\tau}(\eta_{A}^{M},\eta_B^{m})$. Hence, by Corollary \ref{cor:assumption} and the strong Markov property:
{\footnotesize{
\begin{eqnarray*}
    \lefteqn{e^{-r\tau}v^C(X^x_\tau)\1_{\{\tau<\eta_B^{m}\}\cap\{\tau< \gamma_{\tau,A}^{M}\}}}\\&=&e^{-r\tau}\E^{X_\tau^x}[R(\eta_A^{M},\eta_B^{m})\1_{\eta_A^{M}\wedge \eta_B^{m}<\infty}+M_\infty \1_{\eta_A^{M}\wedge \eta_B^{m}=\infty}]\1_{\{\tau<\eta_B^{m}\}\cap\{\tau< \gamma_{\tau,A}^{M}\}}\\
    &=&(
    e^{-r\tau}\E^{x}[(R(\eta_A^{M},\eta_B^{m})\1_{\eta_A^{M}\wedge \eta_B^{m}<\infty})\circ \theta_\tau |\mathcal{F}_\tau]\\
    &&\hspace{43mm}+\E^x[M_\infty\1_{\eta_A^{M}\wedge \eta_B^{m}=\infty}|\mathcal{F}_\tau])\1_{\{\tau<\eta_B^{m}\}\cap\{\tau< \gamma_{\tau,A}^{M}\}}\\
    &=& \big(e^{-r\tau}\E^x[e^{-r(\gamma_{\tau,A}^{M}\wedge \eta_B^{m}-\tau)}(l(X_{\gamma_{\tau,A}^{M}})\1_{\gamma_{\tau,A}^{M}<\eta_B^{m}}+u(X_{\eta_B^{m}})\1_{ \eta_B^{m}<\gamma_{\tau,A}^{M}})\1_{\eta_A^{M}\wedge \eta_B^{m}<\infty}|\mathcal{F}_\tau]
    \\&&\hspace{43mm}+\E^x[M_\infty\1_{\eta_A^{M}\wedge \eta_B^{m}=\infty})|\mathcal{F}_\tau]\big)\1_{\{\tau<\eta_B^{m}\}\cap\{\tau< \gamma_{\tau,A}^{M}\}}\\
    &=& \E^x[R(\gamma^{M}_{\tau,A},\eta_B^{m})|\mathcal{F}_\tau]\1_{\{\tau<\eta_B^{m}\}\cap\{\tau< \gamma_{\tau,A}^{M}\}}. 
\end{eqnarray*}
}}$\square$

[Proof of Lemma~\ref{lem:commonverification}]

Recall that the functions $l$ and $u$ are both bounded. Hence, for any initial value $x$, $l(X),u(X)$ satisfy Assumption \ref{assumption:BSDE} for any fixed $\alpha\in(-\frac{\theta^2}{2},0)$, with $L_\infty=U_\infty=0$. 

By Remark \ref{rem:ItoTanaka} and the property that $V$ is a strong solution of \eqref{HJB1}, $(V(X),V'(X))$ satisfies BSDE (\ref{eq:BSDE}). By Lemma \ref{lem:propertyV} the candidate value function $V$ is bounded, hence the asymptotic condition \eqref{boundary_cond} holds. Therefore $(V(X),V'(X))$ is a solution to the BSDE (\ref{eq:BSDE}). By Theorem \ref{thm:satisfyassC}, $V$ is the value of the game under the common Poisson constraint and $(\eta^C_{(-1,1)}, \eta^C_{(-x^*,-1)\cup(1,x^*)})$ is a saddle point.
$\square$

[Proof of Proposition \ref{prop:existence}]

    We start by considering the following finite horizon BSDE on $[0,k]$ for any $k\geq 0$:
{\footnotesize{
\begin{equation}\label{eq:BSDEn}
    \mathbf{y}_t(k)=\int_t^k(\lambda \max\{l(X_s),\min(\mathbf{y}_s(k),u(X_s))\}-(\lambda+r)\mathbf{y}_s(k))\,ds-\int_t^k Z_s(k)\,dW_s.
\end{equation}
}}
 The driver of BSDE (\ref{eq:BSDEn}) is Lipschitz continuous and of linear growth in $y$, independent of $z$. Hence, by Darling and Pardoux \cite[Proposition 2.3]{darling1997backwards}, there exists a unique pair of solutions $(\mathbf{y}(k),Z(k))\in \mathcal{S}^2_{0,k}\times \mathcal{H}^2_{0,k}$ to the BSDE (\ref{eq:BSDEn}).

Next, we extend BSDE (\ref{eq:BSDEn}) from $[0,k]$ to $[0,\infty)$ by defining $\mathbf{y}_t(k)=Z_t(k)=0$ for $t\geq k$. {Hence we get two sequences $(\mathbf{y}(k))_{k\geq 0}\subset \mathcal{S}^2_{\alpha,\infty}$ and $(Z(k))_{k\geq 0}\subset \mathcal{H}^2_{\alpha,\infty}$. Our next step is to show that these sequences are Cauchy sequences in the corresponding weighted normed spaces. }

Take $k'\geq k\geq 0$. 
It follows from It\^{o}'s formula that we have:
{\footnotesize{
\begin{eqnarray*}
    \lefteqn{e^{2\alpha t}(\mathbf{y}_t(k')-\mathbf{y}_t(k))^2}\\
    &=&-\int_t^{k'} 2ae^{2\alpha s}(\mathbf{y}_s(k')-\mathbf{y}_s(k))^2\,ds\\
    &&+\int_t^{k'} 2e^{2\alpha s}(\mathbf{y}_s(k')-\mathbf{y}_s(k))\lambda(\max\{l(X_s),\min(\mathbf{y}_s(k'),u(X_s))\}-\max\{l(X_s),\min(\mathbf{y}_s(k),u(X_s))\})\,ds\\
    &&-\int_t^{k'} 2e^{2\alpha s}(\lambda+r)(\mathbf{y}_s(k')-\mathbf{y}_s(k))^2\,ds+\int_k^{k'} 2e^{2\alpha s}\lambda l(X_s)(\mathbf{y}_s(k')-\mathbf{y}_s(k))\,ds\\
    &&-\int_t^{k'}2e^{2\alpha s}(\mathbf{y}_s(k')-\mathbf{y}_s(k))(Z_s(k')-Z_s(k))\,dW_s-\int_t^{k'}e^{2\alpha s}(Z_s(k')-Z_s(k))^2\,ds,
\end{eqnarray*}
}}
By an application of the inequality $2ab\leq \frac{a^2}{\delta^2}+\delta^2b^2$, for any $\delta>0$ we have:
{\footnotesize{
\begin{equation}\label{eq:young}
    \int_k^{k'} 2e^{2\alpha s}\lambda l(X_s)(\mathbf{y}_s(k')-\mathbf{y}_s(k))\,ds\leq \int_k^{k'} e^{2\alpha s}\lambda \left(\delta^2 l(X_s)^2+\frac{1}{\delta^2}(\mathbf{y}_s(k')-\mathbf{y}_s(k))^2\right)\,ds.
\end{equation}
}}
Since $\alpha > -r$ we may choose $\delta>0$ such that $\alpha=\frac{1}{2\delta^2}\lambda-r$. Using (\ref{eq:young}) and Lipschitz continuity of the driver, we obtain:
{\footnotesize{
\begin{eqnarray*}
    \lefteqn{e^{2\alpha t}(\mathbf{y}_t(k')-\mathbf{y}_t(k))^2}\\
    &\leq &\int_t^{k'} 2e^{2\alpha s}(\mathbf{y}_s(k)-\mathbf{y}_s(k'))^2\left(-\alpha+\lambda-(\lambda+r)+\frac{1}{2\delta^2}\lambda\right)\,ds+\int_k^{k'} e^{2\alpha s}\lambda \delta^2 l(X_s)^2\,ds\\
    &&-\int_t^{k'}2e^{2\alpha s}(\mathbf{y}_s(k')-\mathbf{y}_s(k))(Z_s(k')-Z_s(k))\,dW_s-\int_t^{k'}e^{2\alpha s}(Z_s(k')-Z_s(k))^2\,ds,
\end{eqnarray*}
}}
and hence, 
{\footnotesize{
\begin{eqnarray}\label{eq:limit}
    \lefteqn{e^{2\alpha t}(\mathbf{y}_t(k')-\mathbf{y}_t(k))^2+\int_t^{k'}e^{2\alpha s}(Z_s(k')-Z_s(k))^2\,ds}\notag\\&\leq& \int_k^{k'} e^{2\alpha s}\lambda \delta^2 l(X_s)^2\,ds-\int_t^{k'}2e^{2\alpha s}(\mathbf{y}_s(k')-\mathbf{y}_s(k))(Z_s(k')-Z_s(k))\,dW_s\notag\\
    &=&\int_k^{k'} e^{2\alpha s}\lambda \delta^2 l(X_s)^2\,ds+\int_0^{t}2e^{2\alpha s}(\mathbf{y}_s(k')-\mathbf{y}_s(k))(Z_s(k')-Z_s(k))\,dW_s\notag\\&&\hspace{33mm}-\int_0^{k'}2e^{2\alpha s}(\mathbf{y}_s(k')-\mathbf{y}_s(k))(Z_s(k')-Z_s(k))\,dW_s.
\end{eqnarray}
}}
It can be shown that $(\int_0^t2e^{2\alpha s}(\mathbf{y}_s(k')-\mathbf{y}_s(k))(Z_s(k')-Z_s(k))\,dW_s)_{t\geq 0}$ is a uniformly integrable martingale. Indeed, by the BDG inequality and the inequality $2ab\leq \frac{a^2}{\delta_1^2}+\delta_1^2b^2$, for any $\delta_1>0$,
{\footnotesize{
\begin{eqnarray*}
    \lefteqn{\E\left[\sup_{t\geq 0}\bigg|\int_0^{t}2e^{2\alpha s}(\mathbf{y}_s(k')-\mathbf{y}_s(k))(Z_s(k')-Z_s(k))\,dW_s\bigg|\right]}\\
&\leq & \frac{C}{\delta_1^2}\E\left[\left(\sup_{0\leq s\leq k'}e^{2\alpha s}(\mathbf{y}_s(k')-\mathbf{y}_s(k))^2\right)\right]+C\delta_1^2\E\left[\left(\int_0^{k'}e^{2\alpha s}(Z_s(k')-Z_s(k))^2\,d s\right)\right]\\
&\leq & \frac{C}{\delta_1^2}\lVert\mathbf{y}(k')-\mathbf{y}(k)\rVert_{\mathcal{S}^2_{\alpha,k'}}+C\delta_1^2\lVert Z(k')-Z(k)\rVert_{\mathcal{H}^2_{\alpha,k'}}<\infty.
\end{eqnarray*}
}}
Taking expectations at $t=0$ {in (\ref{eq:limit})}, we get:
{\footnotesize{
\[{\lVert Z(k')-Z(k)\rVert_{\mathcal{H}^2_{\alpha,\infty}}}=\E\left[\int_0^{k'}e^{2\alpha s}(Z_s(k')-Z_s(k))^2\,ds\right]\leq \lambda \delta^2\E\left[\int_k^{k'} e^{2\alpha s} l(X_s)^2\,ds\right]\downarrow 0\]
}}
as $k,k'\rightarrow \infty$. Hence $(Z(k))_{k\geq 0}$ is a Cauchy sequence in the space $\mathcal{H}^2_{\alpha,\infty}$ and converges to some limit which we denote by $Z=(Z_t)_{t\geq 0}$.

Similarly, taking expectations of the supremum over $t$ {in (\ref{eq:limit})}, we get:
{\footnotesize{
\begin{eqnarray*}
    \lefteqn{\E[\sup_{t\geq 0}e^{2\alpha t}(\mathbf{y}_t(k')-\mathbf{y}_t(k))^2]}\\
    &\leq & \E\left[\int_k^{k'} e^{2\alpha s}\lambda \delta^2 l(X_s)^2\,ds\right]+ \frac{C}{\delta_1^2}\lVert\mathbf{y}(k')-\mathbf{y}(k)\rVert_{\mathcal{S}^2_{\alpha,k'}}+C\delta_1^2\lVert Z(k')-Z(k)\rVert_{\mathcal{H}^2_{\alpha,k'}}.
\end{eqnarray*}
}}

Recall that $\mathbf{y}_t(k)=Z_t(k)=0$ for $t\geq k$, which implies that $\lVert \mathbf{y}(k')-\mathbf{y}(k)\rVert _{\mathcal{S}^2_{\alpha,k'}}=\lVert \mathbf{y}(k')-\mathbf{y}(k)\rVert _{\mathcal{S}^2_{\alpha,\infty}}$ and $\lVert Z(k')-Z(k)\rVert _{\mathcal{H}^2_{\alpha,k'}}=\lVert Z(k')-Z(k)\rVert _{\mathcal{H}^2_{\alpha,\infty}}$. Hence, by choosing $\delta_1$ such that $\delta_1^2>C$, we have:
{\footnotesize{
\[\left(1-\frac{C}{\delta_1^2}\right)\lVert \mathbf{y}(k')-\mathbf{y}(k)\rVert _{\mathcal{S}^2_{\alpha,\infty}}\leq \E\left[\int_k^{k'} e^{2\alpha s}\lambda \delta^2 l(X_s)^2\,ds\right]+ C\delta_1^2\lVert Z(k')-Z(k)\rVert _{\mathcal{H}^2_{\alpha,\infty}}.\]}}
Hence, taking $k,k'\rightarrow \infty$ {and applying the result that $(Z(k))_{k\geq 0}$ is a Cauchy sequence in the space $\mathcal{H}^2_{\alpha,\infty}$}, we have that $(\mathbf{y}(k))_{k\geq 0}$ is a Cauchy sequence in the space $\mathcal{S}^2_{\alpha,\infty}$ and converges to some limit which we denote by $\mathbf{y}=(\mathbf{y}_t)_{t\geq 0}$. 


By considering limits of $\mathbf{y}(k)$ and $Z(k)$ as $k\rightarrow \infty$, it is standard to check that $(\mathbf{y},Z)$ satisfies BSDE (\ref{eq:BSDE}). The process $\mathbf{y}$ is defined as a limit of a Cauchy sequence in the space $\mathcal{S}^2_{\alpha,\infty}$; therefore, it is immediate that $\lim\limits_{t\uparrow \infty}\E[e^{2\alpha t}\mathbf{y}_t^2]=0$. The existence of a solution to BSDE (\ref{eq:BSDE}) thus follows.

For the uniqueness of the solution, let $(\mathbf{y},Z)$ and $(\mathbf{y}',Z')$ be two solutions to BSDE (\ref{eq:BSDE}), and define $\Delta \mathbf{y}=\mathbf{y}-\mathbf{y}'$ and $\Delta Z=Z-Z'$. Then, {for $0\leq t\leq T$}, $\Delta \mathbf{y}$ solves:
{\footnotesize{
\[\Delta \mathbf{y}_t=\Delta \mathbf{y}_T+\int_t^T \left(f(s,\mathbf{y}_s)-f(s,\mathbf{y}'_s)\right)\,ds-\int_t^T\Delta Z_s\,dW_s.\]
}}
Applying It\^{o}'s formula to $e^{2\alpha t}(\Delta \mathbf{y}_t)^2$, a similar argument as in the proof of existence implies that $\lVert \Delta \mathbf{y}\rVert _{\mathcal{S}^2_{\alpha,\infty}}=\lVert \Delta Z\rVert _{\mathcal{H}^2_{\alpha,\infty}}=0$, hence uniqueness follows.
$\square$

[Proof of Lemma \ref{lem:UI}]

The integrability of $Y$ follows from the result that $\mathbf{y}\in \mathcal{S}^2_{\alpha,\infty}$. Assumption~\ref{assumption:BSDE} states integrability of $L$ and $U$, therefore the integrability of $\hat{Y}$ follows.

    {The processes $L=(L_t)_{t\geq 0}$ and $U=(U_t)_{t\geq 0}$ are in $\mathcal{H}^2_{\alpha+r,\infty}$ under Assumption \ref{assumption:BSDE}. This implies that $\int_0^\infty e^{2(\alpha+r)t}L_t^2\,dt$ and $\int_0^\infty e^{2(\alpha+r)t}U_t^2\,dt$ are a.s. finite. We know that, if a nonnegative function $f$ satisfies $\int_0^\infty f(x)\,dx<\infty$ and $\lim\limits_{t\uparrow \infty}f(x)$ exists, then this limit must be zero. Hence, by Assumption \ref{assumption:BSDE}, $L_t,U_t\rightarrow 0$ a.s..

    The process $\mathbf{y}$ is in the space $ \mathcal{S}^2_{\alpha,\infty}$ with $\alpha>-r$, therefore it follows that $Y=(Y_t)_{t\geq 0}\in  \mathcal{S}^2_{\alpha+r,\infty}$, which implies that $Y_t\rightarrow 0$ a.s.. Hence $\hat{Y}_{t}=\max\{\min\{U_{t}, Y_{t}\},L_{t}\}\rightarrow 0$ a.s..}
$\square$



\end{document}